%% file: main-v3.tex
\documentclass[11pt, reqno]{amsart}
\usepackage[margin = 1in]{geometry}
\usepackage{graphicx, amssymb, amsmath, amsthm}
\usepackage[utf8]{inputenc}
\usepackage{slashed}
\usepackage{epsfig}
\usepackage{hyperref}
\usepackage{comment}
\usepackage{mathrsfs}
\numberwithin{equation}{section}
\usepackage{subfigure}
\usepackage{here}
\usepackage{float}
\usepackage{pifont}
\usepackage{enumitem}
\usepackage{mathtools}

\usepackage{tikz}
\usetikzlibrary{positioning}
\usetikzlibrary{arrows.meta}
\tikzset{
  implication/.style={
    double,
    double distance=1.5pt,
    -{Implies}
  }
}

\DeclareMathOperator{\Int}{Int}
\DeclareMathOperator{\End}{End}

\DeclareMathOperator{\sw}{sw}

\DeclareMathOperator{\id}{id}

\DeclareMathOperator{\diver}{div}
\DeclareMathOperator{\proj}{proj}
\DeclareMathOperator{\colim}{colim}
\newcommand{\op}{\mathrm{op}}

\newcommand{\tr}{\operatorname{tr}}
\newcommand{\CP}{\mathbb{CP}}

\newcommand{\N}{\mathbb{N}}

\newcommand{\Q}{\mathbb{Q}}
\newcommand{\R}{\mathbb{R}}

\newcommand{\Z}{\mathbb{Z}}

\newcommand{\bc}{\mathbf{c}}

\newcommand{\cA}{\mathcal A}

\newcommand{\cC}{\mathcal C}

\newcommand{\cG}{\mathcal G}

\newcommand{\cZ}{\mathcal Z}

\newcommand{\fM}{\mathfrak M}
\newcommand{\fs}{\mathfrak s}

\newtheorem{theorem}{Theorem}[section]
\newtheorem{lemma}[theorem]{Lemma}
\newtheorem{corollary}[theorem]{Corollary}

\newtheorem{proposition}[theorem]{Proposition}
\newtheorem{question}[theorem]{Question}

\theoremstyle{definition}
\newtheorem{definition}[theorem]{Definition}
\newtheorem{remark}[theorem]{Remark}
\newtheorem{example}[theorem]{Example}

\begin{document}

\title{Uniformly Positive Scalar Curvature on Manifold Interiors}

\author[Shuli Chen]{Shuli Chen}
\address[Shuli Chen]{School of Mathematics, Institute for Advanced Study, 1 Einstein Dr, Princeton, NJ, 08540, United States}
\email{shulic@ias.edu}

\begin{abstract}
We obtain topological obstructions to the existence of a complete Riemannian metric with uniformly positive scalar curvature on the interiors of certain 4-manifolds using Seiberg--Witten theory. In particular, no punctured blowups of a homotopy K3 manifold admit complete metrics of uniformly positive scalar curvature. 

Further, let $X$ be a compact oriented smooth manifold of dimension $n$. We show that if $n=6,7$ and $\operatorname{Int}X$ admits a complete metric of uniformly positive scalar curvature, then $X$ admits a positive scalar curvature metric which is a product near the boundary. In dimension 5, an analogous result holds under additional hypotheses and after possibly changing the smooth structure on $X$. If $n=5,6,7$ and $\operatorname{Int}X$ admits a complete metric of positive scalar curvature, then $X$ admits a positive scalar curvature metric with minimal boundary. One ingredient is a new PSC deformation result.
\end{abstract}

\maketitle

\section{Introduction}\label{Introduction}

Scalar curvature $R$ is the weakest curvature invariant of a Riemannian manifold $(X, g)$, but the existence of positive scalar curvature (abbreviated as PSC) metrics can place strong topological constraints. In this paper we are concerned with the existence of a complete Riemannian metric with uniformly positive scalar curvature $R_g \ge c >0$ (abbreviated as UPSC) on noncompact manifolds.

The foundational work of Gromov--Lawson \cite{GromovLawson1983Dirac} established many of the basic existence and obstruction results for such metrics. In dimension 3, Gromov \cite{Gromov2023FourLectures} and Wang \cite{Wang2025Topology} independently classified all complete connected orientable 3-manifolds admitting uniformly positive scalar curvature using $\mu$-bubbles. Earlier works of Chang--Weinberger--Yu \cite{ChangWeinbergerYu2010Taming} and Bessi\`eres--Besson--Maillot \cite{BessieresBessonMaillot2011Ricci} also established related results under additional hypotheses. In higher dimensions, various topological obstruction results for such metrics via $\mu$-bubbles were obtained by Chodosh--M\'aximo--Mukherjee \cite{ChodoshMaximoMukherjee2024UniformlyPositive} in dimension 4, by Sweeney \cite{Sweeney2026Positive} in dimension 5, and by Shi--Wang--Wu--Zhu \cite{ShiWangWuZhu2026Open} in dimensions at most 7. We also note that Block--Weinberger \cite{BlockWeinberger1999Arithmetic} gave a classification of finite-volume irreducible locally symmetric spaces admitting complete metrics of uniformly positive scalar curvature in terms of arithmeticity and $\Q$-rank. On the construction side, Das \cite{Das2026Constructions} established certain construction criteria and Chen--Jin--Jovanovic \cite{ChenJinJovanovic2026Positive} constructed complete metrics of uniformly positive scalar curvature on the product of three connected noncompact smooth manifolds. Deng \cite{Deng2026LocallyConformallyFlat} and Sweeney \cite{Sweeney2026Topology} independently constructed, in every dimension $n\geq 6$, a contractible open manifold not homeomorphic to $\R^n$ that admits a complete metric of uniformly positive scalar curvature.

Gromov \cite[Section 3.6.1]{Gromov2023FourLectures} conjectured that a manifold $X$ admits a complete metric of uniformly positive scalar curvature if and only if $X$ admits a complete metric of positive scalar curvature with at most $C$-quadratic decay for any $C > C_n$ where $C_n>0$ is a dimensional constant. This has been verified by the author \cite{Chen2026OptimalDecay} in dimension 3 for orientable manifolds, and subsequently by Wang and Zhu in dimensions 4 through 7 for orientable manifolds \cite{WangZhu2026QuadraticDecay} and then in all dimensions \cite{WangZhu2026Uniform} without the orientability assumption, with the sharp constant $C_n = \frac{n-1}{n}$. In particular, in their first paper, Wang and Zhu \cite[Proof of Theorem 1.3]{WangZhu2026QuadraticDecay} established the following useful UPSC deformation result. The dimension restriction is due to regularity of the $\mu$-bubbles used in the exhaustion. 
\begin{theorem}[\cite{WangZhu2026QuadraticDecay}]\label{thm: Wang--Zhu}
Let $4 \le n \le 7$. Suppose $X^n$ is a connected orientable noncompact $n$-manifold that admits a complete metric of uniformly positive scalar curvature. Then $X$ admits an exhaustion by compact smooth codimension-zero submanifolds $K_1 \subset K_2 \subset \dots \subset X$ and a complete metric $g$ of uniformly positive scalar curvature such that near each hypersurface $\Sigma_i \coloneqq  \partial K_i$, the metric is of product form $g = g_{\Sigma_i} + dt^2$.
\end{theorem}

In this paper, combining Theorem \ref{thm: Wang--Zhu} with Seiberg--Witten theory for cylindrical manifolds, we obtain the following new topological obstruction result in dimension 4:

\begin{theorem}\label{thm: dim 4}
Let $M$ be a closed, connected, oriented, smooth $4$-manifold with $b_2^+(M)>1$, and suppose the Seiberg--Witten invariant $\sw_M(\mathfrak s)\neq 0$ for some $Spin^{c}$ structure $\fs$. Suppose $X \subset M$ is a compact connected smooth codimension-zero submanifold with $b_2^+(X) > 0$. Then $\Int X$ doesn't admit a complete metric of uniformly positive scalar curvature.
\end{theorem}

Here recall that for a compact oriented $4$-manifold $X$, $b_2^+(X)$ (respectively $b_2^-(X)$) is the dimension of a maximal positive-definite (respectively negative-definite) subspace of $H_2(X;\mathbb R)$ with respect to the intersection form. The above result applies in particular to $X = M \setminus \cup_i B_i$ where $B_i$'s are disjoint open balls. We note that $b_2^+(X) = b_2^+(M)$ (see Lemma \ref{lem: puncture} for a simple proof). Then $\Int X$ is diffeomorphic to the punctured manifold $M \setminus \{p_1, \dots, p_k\}$. 

In \cite{ChodoshMaximoMukherjee2024UniformlyPositive}, Chodosh, M\'aximo, and Mukherjee showed that for some smooth structure, the once-punctured K3 manifold does not admit a complete metric of uniformly positive scalar curvature. Here, we can extend this nonexistence result to the puncture of any closed smooth homotopy K3 manifold. If $M$ is a closed smooth 4-manifold homotopy equivalent to a K3 manifold and oriented so that $\sigma(M)=-16$, then $b_2^+(M) = 3$ and it has a $Spin^c$-structure such that the Seiberg--Witten invariant is nonzero \cite[Theorem 1.1]{MorganSzabo1997HomotopyK3}. Therefore we obtain 

\begin{corollary}\label{cor: K3}
Let $M$ be a closed smooth oriented 4-manifold which is homotopy equivalent to a K3 surface. Then $M \setminus \{p\}$ does not admit a complete metric of uniformly positive scalar curvature. 
\end{corollary}
\begin{remark}
During the preparation of the first draft, Kumar and Zhang \cite[Theorem A]{KumarZhang2026Spin} showed that a closed spin four-manifold with nonzero $\hat{A}$-genus, even after being
punctured, does not admit a complete metric of uniformly positive scalar curvature. In particular, their result also gives Corollary \ref{cor: K3}.   
\end{remark}

Take any smooth homotopy K3 manifold $K$. Choose the orientation of $K$ so that $b_2^+(K) = 3$. For any $r \ge 1$, consider its $r$-fold blowup $M = K\#_r\overline{\CP}^{2}$. $M$ is nonspin, and homeomorphic (but not diffeomorphic to) $\#_3 \CP^2 \# (\#_{19+r}\overline{\CP}^2)$. We have $b_2^+(M) = b_2^+(K) = 3$ since blowups don't change $b_2^+$. The blowup formula (see e.g. \cite[Theorem 4.6.7]{Nicolaescu2000SeibergWitten}) then gives nonzero Seiberg--Witten invariants on $M$. Thus Theorem \ref{thm: dim 4} gives the following corollary not implied by \cite{KumarZhang2026Spin}:
\begin{corollary}
Take any smooth homotopy K3 manifold $K$ and for $r \ge 1$ consider its $r$-fold blowup $M = K\#_r\overline{\CP}^{2}$. Then $M \setminus \{p\}$ does not admit a complete metric of uniformly positive scalar curvature.  
\end{corollary}

On the other hand, take $M=K3\#(S^2\times S^2)$. This is a spin manifold with $\hat{A}(M)=-\frac{\sigma(M)}{8} =2\ne 0$. Thus \cite[Theorem A]{KumarZhang2026Spin} implies $M \setminus \{p\}$ does not admit a complete metric of uniformly positive scalar curvature. However, since both summands have positive $b_2^+$, the connected sum theorem (see e.g. \cite[Theorem 4.6.1]{Nicolaescu2000SeibergWitten}) shows all the Seiberg--Witten invariants of $M$ vanish, so our Theorem \ref{thm: dim 4} doesn't imply this result. 

\begin{remark}
Note that such obstructions on punctured manifolds don't extend to connected sums in general, even when the summands are closed. Indeed, let $M$ be an oriented K3 manifold with the standard smooth structure, and we write $M^{\op}$ for $M$ with the opposite orientation. Then $M \# M^{\op} \cong \#_{22} (S^2 \times S^2)$ (see e.g. \cite[p. 558]{Reutter2023Semisimple}), and hence admits a metric of positive scalar curvature. 
\end{remark}

More generally, for a compact oriented manifold $(X^n, \partial X)$, we can investigate the relationship between several natural positive scalar curvature conditions:

\begin{enumerate}
\item \label{cond: product} $X$ admits a positive scalar curvature metric which is a product metric in a collar neighborhood of $\partial X$,
\item \label{cond: minimal} $X$ admits a positive scalar curvature metric for which $\partial X$ is minimal,
\item \label{cond: upsc} $\Int X$ admits a complete metric of uniformly positive scalar curvature $R \ge 1$,
\item \label{cond: psc} $\Int X$ admits a complete metric of positive scalar curvature $R > 0$.
\end{enumerate}

In general, the following implications are immediate: $\eqref{cond: product} \implies \eqref{cond: minimal}$ and $\eqref{cond: product} \implies \eqref{cond: upsc} \implies \eqref{cond: psc}$. 
We also note that Condition \eqref{cond: minimal} is known to be equivalent to many other conditions including the existence of a positive scalar curvature metric on $X$ with mean convex boundary, and a positive scalar curvature metric on $X$ with totally geodesic boundary \cite{BarHanke2023BoundaryConditions}. These lead naturally to the following questions.

\begin{question}\label{qn: 1}
Does Condition \eqref{cond: upsc} imply Condition \eqref{cond: product}?    
\end{question}

\begin{question}\label{qn: 2}
Are there any relationships between Conditions \eqref{cond: minimal} and \eqref{cond: psc}? 
\end{question}

We note that Question \ref{qn: 1} has a positive answer in dimension 3 by combining the surgery result of Gromov--Lawson \cite{GromovLawson1980Classification} and Schoen--Yau \cite{SchoenYau1979Structure} and the aforementioned classification result of 3-manifolds with a complete metric of uniformly positive scalar curvature of Gromov \cite{Gromov2023FourLectures} and Wang \cite{Wang2025Topology}.

For Question \ref{qn: 2}, we note that in dimension 3, \cite[Corollary 6.5]{CarlottoLi2024Constrained} gives a complete classification of compact $(X,\partial X)$ satisfying Condition \eqref{cond: minimal}. Then in dimension 3, the implication $\eqref{cond: psc} \implies \eqref{cond: minimal}$ follows from a result of Schoen--Yau \cite[Theorem 4]{SchoenYau1982CompleteThreeManifolds} stating the fundamental group of a 3-manifold with complete PSC metric cannot contain a surface subgroup.
However, the converse fails: for $g \ge 2$, the interior of a genus $g$ handlebody doesn't admit a complete metric of positive scalar curvature by the recent work of You and Zhang \cite{YouZhang2026Rigidity}, whereas the compact handlebodies admit a metric of positive scalar curvature with minimal boundary. To the best of my knowledge, in dimension at least 4, there is no known example of a manifold satisfying Condition \eqref{cond: minimal} but not \eqref{cond: psc}, or the other way around.

In this paper, we give an affirmative answer to Question \ref{qn: 1} in dimensions 6 and 7. 
The proof is a combination of Theorem \ref{thm: Wang--Zhu} and high dimensional surgery theory.

\begin{theorem}\label{thm: UPSC dim 6, 7}
Let $n = 6, 7$. Let $(X, \partial X)$ be a connected compact oriented manifold of dimension $n$ with non-empty boundary $\partial X$. Suppose $\Int X$ admits a complete metric of uniformly positive scalar curvature. Then $X$ admits a positive scalar curvature metric which is a product near the boundary. 
\end{theorem}

This can be viewed as an improved version of 
\cite[Corollary 3.4]{ChodoshMaximoMukherjee2024UniformlyPositive} in dimensions 6 and 7, which shows $\partial X$ admits a Riemannian metric with positive scalar curvature under the same assumption.

Further, we give a partial answer to Question \ref{qn: 2} in dimensions 5, 6 and 7. 
\begin{theorem}\label{thm: PSC dim 5, 6, 7}
Let $n = 5, 6, 7$. Let $(X, \partial X)$ be a connected compact oriented manifold of dimension $n$ with non-empty boundary $\partial X$. Suppose $\Int X$ admits a complete metric of positive scalar curvature. Then $X$ admits a positive scalar curvature metric with minimal boundary. 
\end{theorem}

The above theorem relies on the following PSC deformation result analogous to Theorem \ref{thm: Wang--Zhu} which might be of independent interest. 

\begin{proposition}\label{prop: PSC deformation}
Let $3 \le n \le 7$. Suppose $X^n$ is a connected orientable noncompact $n$-manifold that admits a complete metric of positive scalar curvature. Then $X$ admits an exhaustion by compact smooth codimension-zero submanifolds $K_1 \subset K_2 \subset \dots \subset X$ such that for each $i\ge 1$, $K_i$ admits a metric $g_i$ such that $R_{g_i} > 0$ on $K_i$ and $H_{g_i}(\partial K_i) = 0$.
\end{proposition}
In the above proposition, note that in general we cannot obtain one single complete PSC metric on $X$ where all $\partial K_i$ are minimal, because that would imply $X$ admits a complete UPSC metric by \cite[Theorem 1.3]{Das2026Constructions}.

Many previous obstructions to complete PSC metrics can actually be obtained via this deformation. For example, we can obtain an alternative proof of Theorem 1.2 in \cite{ChenChuZhu2023Aspherical} as a corollary.

\begin{corollary}[{\cite[Theorem 1.2]{ChenChuZhu2023Aspherical}}] \label{cor: aspherical}
Let $N^n$, $n \in \{3,4,5\}$, be a closed aspherical manifold and let $X^n$ be an arbitrary $n$-manifold. Then the connected sum $Y = N \# X$ admits no complete metric of positive scalar curvature.    
\end{corollary}

In dimension 5, the counterexample in \cite[Remark 1.25]{Rosenberg2007ProgressReport} shows Question \ref{qn: 1} has a negative answer in general. Let $Y\subset\CP^3$ be a smooth complex hypersurface of odd degree $d\ge5$. Then $Y$ is simply connected and nonspin, with $b_2^+(Y)>1$ and a nonzero Seiberg--Witten invariant. Consequently, $Y$ does not admit a metric of positive scalar curvature, whereas $Y\times S^1$ does. Set $X=Y\times[-1,1]$. Lifting a positive scalar curvature metric on $Y\times S^1$ to the covering $Y\times\R\cong\Int X$ gives a complete metric of uniformly positive scalar curvature on $\Int X$. However, $X$ cannot admit a positive scalar curvature metric that is a product near the boundary, since $\partial X = Y \sqcup Y^{\op}$ does not admit a metric of positive scalar curvature. 

On the other hand, the works of Chodosh--M\'aximo--Mukherjee \cite{ChodoshMaximoMukherjee2024UniformlyPositive} and Kumar--Sen \cite{KumarSen2025ExoticStructures} suggest that the issue in dimension 5 might be only due to exotic smooth structures in dimension 4. 
In this spirit, using surgery theory, we prove that under additional hypotheses in dimension 5, we can give a positive partial answer to Question \ref{qn: 1}, if we allow a change of smooth structure. Recall that an orientable manifold is \emph{totally nonspin} if its universal cover does not admit a spin structure.

\begin{theorem}\label{thm: dim 5 homeo}
Let $(X, \partial X)$ be a compact connected oriented $5$-manifold with non-empty boundary $\partial X$. Assume $X$ is either spin or totally nonspin, and assume each component of $\partial X$ is simply connected.
Suppose $\Int X$ admits a complete metric of uniformly positive scalar curvature. Then there exists a smooth compact $5$-manifold $X'$, homeomorphic to $X$, such that $X'$ admits a positive scalar curvature metric which is a product near the boundary.  
\end{theorem}

We summarize the relationship between the curvature conditions in Figure \ref{fig: curvature-conditions}.

\begin{figure}[ht]
    \centering
    \input{conditions.tex}
    \caption{Relations among the curvature conditions.}
    \label{fig: curvature-conditions}
\end{figure}
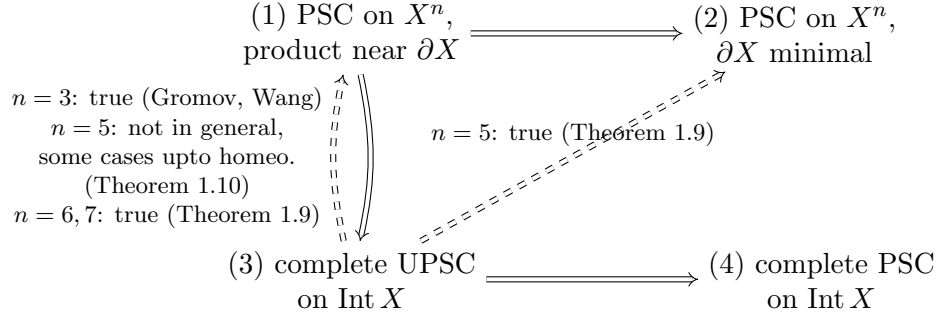

This paper is organized as follows. In Section \ref{sec: PSC deformation} we prove the PSC deformation result Proposition \ref{prop: PSC deformation}, and prove Corollary \ref{cor: aspherical}. In Section \ref{sec: dim 4} we focus on dimension 4. We first review basics of Seiberg--Witten theory, give some results on Seiberg--Witten theory for cylindrical manifolds, and prove Theorem \ref{thm: dim 4}. In Section \ref{sec: dim 5,6,7} we focus on dimensions 5, 6, 7. We first review basics of bordism theory and the language of tangential structures, and then prove Theorem \ref{thm: UPSC dim 6, 7}, Theorem \ref{thm: PSC dim 5, 6, 7}, and Theorem \ref{thm: dim 5 homeo}.

\subsection*{Acknowledgment}
The author thanks Otis Chodosh, Demetre Kazaras, Anubhav Mukherjee, and Qianhe Qin for discussions about the punctured K3 problem. This work is carried out while the author is a postdoctoral member at the Institute for Advanced Study, and the author thanks the Ambrose Monell Foundation for providing funding for the membership. The author also thanks the Institute for Advanced Study for providing access to ChatGPT Education Pro.

\subsection*{AI Use} The author used ChatGPT to help study and locate relevant references for Seiberg--Witten theory and bordism theory, and to proofread the manuscript. In particular, ChatGPT suggested following the approach of \cite{Kato2021HarmonicTheory}, and the idea of the proof of Theorem \ref{thm: dim 5 homeo}, in the case where the manifold $X$ is totally nonspin, was due to ChatGPT. The idea that we can extend Wang--Zhu's method to the PSC case as in Section \ref{sec: PSC deformation} was also pointed out by ChatGPT. This paper is not written by AI.

\section{PSC deformation}\label{sec: PSC deformation}
In this section we will follow the notation used in \cite{WangZhu2026QuadraticDecay}. For a Caccioppoli set $F$ and an open set $U$, write
\[
P(F;U)=|D\chi_F|(U),
\]
and denote its reduced boundary by $\partial^*F$.
By a smooth band we mean a compact smooth codimension-zero submanifold $A$ whose boundary $\partial A$ is partitioned into two unions of components, $\partial_{\mathrm{in}} A$ and $\partial_{\mathrm{out}} A$, each nonempty.

For an oriented hypersurface with unit normal $\nu$, our sign convention is
\[
\mathrm{II}(X,Y)=\langle\nabla_X\nu,Y\rangle,
\qquad H=\operatorname{tr}\mathrm{II}.
\]

\begin{definition}[minimizing $\mu$-bubble]\label{def: minimizing-mu-bubble}
Let $A$ be a compact smooth band with
$\partial A=\partial_{\mathrm{in}}A\sqcup\partial_{\mathrm{out}}A$,
and let $h\in C^\infty(A^\circ)$.
Fix a smooth reference region $\widehat{\Omega}\subset A$ which contains a relative collar of $\partial_{\mathrm{in}}A$ and is disjoint from a relative collar of $\partial_{\mathrm{out}}A$.

A smooth hypersurface $\Sigma=\partial\Omega\cap A^\circ$ is a \emph{minimizing $\mu$-bubble for $h$ in $A$} if $\Omega$ is a smooth global minimizer of
\begin{equation}
\mathcal{A}_h(F)
=
P(F;A^\circ)
-
\int_{A^\circ}
(\chi_F-\chi_{\widehat{\Omega}})h\,dV_g
\end{equation}
among all Caccioppoli sets $F$ with $F\Delta\widehat\Omega\Subset A^\circ$.
\end{definition}

Similar to the proof of Theorem \ref{thm: Wang--Zhu} in \cite{WangZhu2026QuadraticDecay}, to obtain the PSC deformation result, we first need to obtain an exhaustion via $\mu$-bubbles, and then apply a conformal change. Here we have the following PSC exhaustion result. It is very similar to the exhaustion result as in \cite{Chen2026OptimalDecay} (cf. \cite[Proposition 2.2]{WangZhu2026QuadraticDecay}). The only difference is that now we only have positive scalar curvature instead of some control on the decay rate of the scalar curvature at infinity; as a result we can only control the negative part of the prescribing function as in property \eqref{item:band-potential} below.

\begin{proposition}
\label{prop: PSC exhuastion}
Let $3 \le n \le 7$. Suppose $M$ is a complete noncompact orientable $n$-manifold with positive scalar curvature.

Given a compact set $K_0\Subset M$, there are
\begin{itemize}
\item pairwise disjoint compact smooth bands $A_i\Subset M$;
\item smooth functions $h_i\colon A_i^\circ\to\mathbb R$;
\item smooth Caccioppoli sets $\Omega_i\subset A_i$;
\item smooth embedded closed hypersurfaces
  $$
    \Sigma_i=\partial\Omega_i\cap A_i^\circ;
  $$
\item a compact smooth exhaustion $\{K_i\}_{i\ge1}$ of $M$ by
codimension-zero submanifolds, with
  $$
    K_i\Subset\operatorname{Int}K_{i+1}\quad(i\ge0),
    \qquad M=\bigcup_{i=1}^\infty K_i,
    \qquad \partial K_i=\Sigma_i,
  $$
\end{itemize}

with the following properties.

\begin{enumerate}[label=\textup{(\roman*)},ref=\roman*]
\item\label{item:band-potential}
Writing \(h_i^-=\max\{-h_i,0\}\), one has
\begin{equation}\label{eq:band-potential}
R_g+\frac{n}{n-1}h_i^2-2|\nabla h_i|>0 \qquad\text{on }\operatorname{supp}_{A_i^\circ}(h_i^-).
\end{equation}
\item\label{item:band-collars}
Writing
$\partial A_i=\partial_{\rm in}A_i\sqcup
\partial_{\rm out}A_i$, one has
$$
  h_i\longrightarrow+\infty
  \quad\text{at }\partial_{\rm in}A_i,
  \qquad
  h_i\longrightarrow-\infty
  \quad\text{at }\partial_{\rm out}A_i.
$$
\item\label{item:band-minimizer}
The set $\Omega_i$ is a minimizing region for $h_i$ in $A_i$
in the sense of Definition \ref{def: minimizing-mu-bubble}.

\item\label{item:band-variation}
If $\nu_i$ is the outward unit normal of $\Omega_i$,
then
$$
  H_{\Sigma_i}=h_i.
$$

\item\label{item:band-normal-orientation}
For every $i$, the outward unit normal of $K_i$ along $\Sigma_i$ is $\nu_i$.
\end{enumerate}

\noindent
The bands, hypersurfaces, and blocks need not be connected.  Statements
concerning disconnected objects are interpreted componentwise.
\end{proposition}
\begin{proof}
This is essentially the same as in \cite{Chen2026OptimalDecay} (cf. \cite[Proposition 2.2]{WangZhu2026QuadraticDecay}); the only difference is that the choice of $h$ will be different due to different curvature assumptions. A similar choice of $h$ is used in other works dealing with positive scalar curvature, e.g. \cite[Proof of Theorem 3]{ChodoshLi2024SoapBubbles}.
We briefly sketch the construction here. 

By the Greene--Wu smoothing theorem
\cite[Proposition 2.1]{GreeneWu1979Approximations}, by smoothing and rescaling a distance function, we can choose a smooth proper exhaustion function $\rho$ on $M$ such that $|\nabla\rho|\le 1$.

We first construct a band and the prescribed function outside an arbitrary compact set $K \Subset M$. Choose $a >\max_K\rho+\frac{\pi}{2}$ so that $a-\frac{\pi}{2}$ is a regular value.

The compactness of \(\{a-\frac{\pi}{2}\le\rho\le a+1\}\) and positivity of the scalar curvature give
$$\kappa \coloneqq \min_{\{a -\frac{\pi}{2} \le\rho\le a+1\}}R_g>0.
$$

Take $0< \varepsilon < \min\{1, \kappa/6\}$ such that $b \coloneqq a+  \frac{\pi}{4} +\frac{2}{\varepsilon}$ is regular. We define the model function $f\colon (a -\frac{\pi}{2}, b) \to \R$ as
\[
f(t)=
\begin{cases}
\varepsilon\tan (t-a),& a -\frac{\pi}{2}<t\le a+ \frac{\pi}{4},\\
\dfrac{\varepsilon}
{1-\frac{\varepsilon}{2}(t-a - \frac{\pi}{4})},
&a+ \frac{\pi}{4} \le t<b.
\end{cases}
\]

Since $|f'| \le \varepsilon \sec^2(\frac{\pi}{4}) = 2\varepsilon$ on $[a,a+\pi/4)$, and $f'((a+\frac{\pi}{4})^+) = \frac{\varepsilon^2}{2} < 2\varepsilon$,
we can take a smoothing $\tilde{f}$ that is still monotonically increasing and agrees with $f$ outside $(a+\frac{\pi}{8}, a+\frac{\pi}{4})$, and $|\tilde{f}'|\le 3\varepsilon$ on $[a,a+\frac{\pi}{4}]$.

Set
$$\quad A=\{a -\frac{\pi}{2} \le\rho\le b\},
 \qquad
  \partial_{\rm in}A=\{\rho= a -\frac{\pi}{2} \},
  \qquad
  \partial_{\rm out}A=\{\rho=b\}, 
  \qquad
  h=-\tilde{f}(\rho).
$$

The function $h$ tends to $+\infty$ at the inner boundary and to $-\infty$ at the outer boundary. The support of its negative part is contained in $A^\circ\cap\{\rho\ge a\}$. On $\{a \le \rho \le a + \frac{\pi}{4} \}$,
\[
R_g+\frac n{n-1}h^2-2|\nabla h|
\ge \kappa -2 |\tilde{f}'||\nabla \rho| \ge  \kappa - 6 \varepsilon >0,
\]
while on $\{a+\frac{\pi}{4}\le\rho<b\}$,
\[
R_g+\frac{n}{n-1}h^2-2|\nabla h|
\ge R_g +\frac{n}{n-1}\tilde{f}^2-2|\tilde{f}'||\nabla \rho| \ge R_g + \frac{1}{n-1}\tilde{f}^2 >0.
\]
This gives exactly the required inequality on \(\operatorname{supp}_{A^\circ}(h^-)\).

To get the $A_i$ and $h_i$'s, we repeat this construction recursively. For the given compact set $K_0\Subset M$, choose $a_1 >\max_{K_0}\rho+\frac{\pi}{2}$ so that $a_1-\frac{\pi}{2}$ is a regular value, and then carry out the above construction. After constructing the $i$-th band choose $a_{i+1} >b_i+\frac{\pi}{2}$ so that $a_{i+1}-\frac{\pi}{2}$ is a regular value.

This way we obtain pairwise disjoint compact smooth bands $A_i=\{a_i -\frac{\pi}{2} \le\rho\le b_i\}$ with $a_i\to\infty$, and the functions $h_i$ satisfying properties \eqref{item:band-potential}, \eqref{item:band-collars}.

Choose a regular value $\ell_i\in(a_i-\frac{\pi}{2},b_i)$ and set
$$
  \widehat\Omega_i=\{a_i-\frac{\pi}{2}\le\rho\le\ell_i\}\subset A_i.
$$
By \cite[Proposition~2.1]{Zhu2021WidthEstimate}
(see also \cite[Proposition~12]{ChodoshLi2024SoapBubbles}), there is a global minimizer
$\Omega_i$ with
$$
  \Omega_i\Delta\widehat\Omega_i\Subset A_i^\circ,
$$
and for $n\le7$,
$\Sigma_i=\partial\Omega_i\cap A_i^\circ$ is a smooth embedded hypersurface.  We take the corresponding smooth open
representative of $\Omega_i$.
The first variation formula
gives $H_{\Sigma_i}=h_i$.  

Finally, define
$$
  K_i=\overline{\{\rho<a_i -\frac{\pi}{2}\}\cup
  (\Omega_i\cap A_i^\circ)}.
$$
The collar behavior gives $\partial K_i=\Sigma_i$, while properness of
$\rho$, the inequalities $b_i<a_{i+1}-\frac{\pi}{2}$, and $a_i\to\infty$ give compactness,
nesting, and exhaustion.  Near $\Sigma_i$, the chosen open representative
shows that $K_i$ lies on the $\Omega_i$-side of $\Sigma_i$.  Hence its
outward normal along $\Sigma_i$ is $\nu_i$, completing the proof.
\end{proof}

The exhaustion gives us compact manifolds with boundary, and for the later deformation let us specify minimization using competitors that stay inside the manifold with boundary $(W, \partial W)$. If $V\subset W$ is a compact smooth neighborhood of a union
$\Sigma$ of boundary components, write
\[
 V\cap\partial W=\Sigma,\qquad
 \partial V=\Sigma\sqcup S,\qquad S\subset\Int W.
\]

\begin{definition}[one-sided minimizing $\mu$-bubble]
\label{def: one-sided-bubble}
Let $V$, $\Sigma$, and $S$ be as above, and let $h\in C^\infty(V)$ be smooth up to the boundary. We say that $\Sigma$ is a \emph{one-sided minimizing $\mu$-bubble for $h$ in $V$} if $H_g=h$ on $\Sigma$, with respect to the outward normal of $W$, and
\begin{equation}\label{eq: one-sided-minimization}
  \mathcal A_{h,V}(V)\le \mathcal A_{h,V}(F)
\end{equation}
for every finite-perimeter set $F\subset V$ that agrees with $V$ near $S$, where
\begin{equation*}
 \mathcal A_{h,V}(F)
 :=P(F;V)
        -\int_V(\chi_F-1)h\,dV_g.
\end{equation*}
Here $P(F;V)$ includes the boundary contributions, as in \cite[Definition 3.1]{WangZhu2026QuadraticDecay}.
\end{definition}

The following crucial duality statement in \cite{WangZhu2026QuadraticDecay} states that one-sided minimizers can be certified by a bounded divergence-measure field. This can be thought of as a generalization of the existence of smooth foliations near strictly stable one-sided minimizers. 
\begin{lemma}[{\cite[Lemma 3.3]{WangZhu2026QuadraticDecay}}]\label{lem: calibration}
Let $V$, $\Sigma$, and $S$ be as above, let
$h\in C^\infty(V)$, and assume the minimizing inequality \eqref{eq: one-sided-minimization}.
Let $U\subset V$ be a compact smooth neighborhood of $\Sigma$ with
\[
 \partial U=\Sigma\sqcup T,\qquad
 T\subset V^\circ,\qquad U\cap S=\varnothing.
\]
Then there is $Z\in L^\infty(TU)$ satisfying
\[
 |Z|\le1,\qquad \diver Z\in L^\infty(U),\qquad
 \diver Z\le h\ \text{a.e. on }U,\qquad
 [Z,\nu_W]=1\ \text{a.e. on }\Sigma.
\]
Here for a bounded divergence-measure field,
$[Z,\nu]$ denotes its normal trace.
\end{lemma}

On the connected compact manifold with boundary $(W, \partial W,g)$, we consider the conformal Laplacian and conformal Robin operator
\begin{equation*}\label{eq: conformalrobin op}
\begin{cases}
L_g=-\frac{4(n-1)}{n-2}\Delta_g+R_g & \text{in } W, \\
B_g=\partial_\nu+\frac{n-2}{2(n-1)}H_g  & \text{on } \partial W, 
\end{cases} 
\end{equation*}
where $\Delta=\diver\nabla$, and $H_g$ is the mean curvature of $\partial W$ with respect to the outward unit normal $\nu$.

The associated conformal Robin quadratic form is
\begin{equation}\label{eq: conformal Robin form}
E_g(u) \coloneqq \int_W \left(\frac{4(n-1)}{n-2}|\nabla u|^2+ R_{g} u^2\right)\,dV_g
+2\int_{\partial W}H_{g}u^2\,d\sigma.    
\end{equation}

If $E_g(u)>0$ for every nonzero $u\in H^1(W)$, then the first
eigenvalue $\lambda_1$ of $(L_g,B_g)$ is positive.
By standard elliptic theory, there exists a smooth positive
eigenfunction $u$ satisfying
\[
\begin{cases}
L_g u=\lambda_1 u & \text{in } W,\\
B_g u=0 & \text{on } \partial W.
\end{cases}
\]
The conformally deformed metric $\tilde g=u^{4/(n-2)}g$ then satisfies
\[
R_{\tilde g}=\lambda_1 u^{-4/(n-2)}>0,
\qquad
H_{\tilde g}=0.
\]
Thus, positivity of $E_g$ yields a metric of positive scalar curvature with minimal boundary in the conformal class of $g$.

Now we are ready to prove Proposition \ref{prop: PSC deformation}. The proof proceeds the same way as in \cite[Section 3]{WangZhu2026QuadraticDecay}. The only difference is that the control on $h^-$ as in property \eqref{item:band-potential} is enough for our purpose, since we only need to deform to minimal boundary and do not need the boundary to be Yamabe positive itself.

\begin{proof}[Proof of Proposition \ref{prop: PSC deformation}]
Suppose $M$ is a complete noncompact orientable $n$-manifold with positive scalar curvature.
Then we get a compact smooth exhaustion $\{K_i\}_{i\ge1}$ of $M$ by codimension-zero submanifolds, bands $A_i$, and smooth functions $h_i \colon A_i^\circ \to \R$ with properties as in Proposition \ref{prop: PSC exhuastion}.
As discussed above, we only need to show for each $K_i$, $E_g(u) >0$ for every nonzero $u \in H^1(K_i)$.

Let $\Sigma_i \coloneqq \partial K_i$ and denote the outward unit normal of $K_i$ along $\Sigma_i$ by $\nu_i$. 

Let $c_i$ be a regular value of $\rho$ satisfying $a_i-\frac{\pi}{2} < c_i<\min\{a_i,\min_{\Sigma_i}\rho\}$.
Let $V_i=K_i\cap\{\rho\ge c_i\}$. Then $V_i \subset  A_i^\circ \cap K_i$ with
\[
\partial V_i=\Sigma_i\sqcup S_i,
 \qquad S_i\subset\Int K_i.
\]
By properties \eqref{item:band-minimizer},\eqref{item:band-collars}, this choice of $V_i$ ensures that $\Sigma_i$ is a one-sided minimizing $\mu$-bubble for $h_i$ in $V_i$, and $h_i \ge 0$ on a neighborhood of $S_i$.

We further trim a small collar of $S_i$ to obtain a compact smooth neighborhood $U_i \subset V_i$ of $\Sigma_i$ such that
\[
 \partial U_i=\Sigma_i\sqcup T_i,
 \qquad U_i\cap S_i=\varnothing,
\]
and $U_i$ contains the support in $V_i$ of
$h_i^-=\max\{-h_i,0\}$ in its relative interior in $K_i$. In particular, $h_i^-$ vanishes near $T_i$.

By Lemma \ref{lem: calibration}, there is a field $Z_i \in L^\infty(TU_i)$ on $U_i$
with $|Z_i|\le1$, $\diver Z_i \in L^\infty(U_i)$, $\diver Z_i\le h_i$, and $[Z_i,\nu_i]=1$ on $\Sigma_i$. 

Define $X_i=-\frac{(n-2)}{2(n-1)}h_i^-Z_i$ on $U_i$ and extend it by zero to $K_i$. Since $h_i^-$ vanishes near $T_i$, this extension introduces no jump across $T_i$, so $\diver X_i \in L^\infty(K_i)$.

On $\{h_i<0\}$,
we have $X_i = \frac{(n-2)}{2(n-1)}h_iZ_i$, and we compute
$$\diver X_i = \frac{(n-2)}{2(n-1)}\bigl(\langle\nabla h_i,Z_i\rangle
 +h_i\diver Z_i\bigr) \ge \frac{(n-2)}{2(n-1)}\bigl(-|\nabla h_i|+h_i^2\bigr), \qquad |X_i|^2\le \frac{(n-2)^2}{4(n-1)^2}h_i^2,$$
 so
$$
 R_g+\frac{4(n-1)}{n-2}(\diver X_i-|X_i|^2)
 \ge R_g+\frac{n}{n-1}h_i^2-2|\nabla h_i| > 0.
$$
Here positivity follows from property \eqref{item:band-potential}. Since the support $\operatorname{supp}_{V_i}(h_i^-)$ is compact, there exists $\delta_0 > 0$ such that $R_g+\frac{n}{n-1}h_i^2-2|\nabla h_i| \ge \delta_0$ on $\operatorname{supp}_{V_i}(h_i^-)$.

Outside of $\{h_i<0\}$, $X_i$ and its divergence vanish almost everywhere, so $R_g+\frac{4(n-1)}{n-2}(\diver X_i-|X_i|^2) \ge \min_{K_i} R_g > 0$ almost everywhere.

Let $\delta = \min \{\delta_0, \min_{K_i} R_g\} >0$. Then
\begin{equation}\label{eq: block-field-potential}
 R_g+\frac{4(n-1)}{n-2}(\diver X_i-|X_i|^2) \ge \delta \text{ a.e. on $K_i$.}    
\end{equation}

On $\Sigma_i$, the equality
$H_g=h_i$ from property \eqref{item:band-variation} and the trace condition $[Z_i,\nu_i]=1$ give
\begin{equation}\label{eq: block-field-boundary}
 2H_g-\frac{4(n-1)}{n-2}[X_i,\nu_i]
 =2h_i+2h_i^-=2h_i^+\ge 0.
\end{equation}

Then by \cite[Lemma A.3]{WangZhu2026QuadraticDecay}, for every $0 \not \equiv u\in H^1(K_i)$,
\begin{align*}
  E_g(u)
  ={}&\frac{4(n-1)}{n-2}\int_{K_i}|\nabla u+X_i u|^2\,dV_g\\
  &+\int_{K_i}\bigl(R_g+\frac{4(n-1)}{n-2}
    (\diver X_i-|X_i|^2)\bigr)u^2\,dV_g \\
  &+\int_{\partial K_i}\bigl(2H_g-\frac{4(n-1)}{n-2}
    [X_i,\nu_i]\bigr)u^2\,dA_g \\
\ge & \int_{K_i}\delta u^2\,dV_g  &\text{by \eqref{eq: block-field-potential}, \eqref{eq: block-field-boundary}}\\
> & \, 0 &\text{as desired.}
\end{align*}
\end{proof}
 
\begin{proof}[Proof of Corollary \ref{cor: aspherical}]
Let $N^n$, $n \in \{3,4,5\}$, be a closed aspherical manifold and let $X^n$ be an arbitrary $n$-manifold.  
By taking orientable covers we can assume without loss of generality that $N$ and $X$ are both orientable. The case where $X$ is closed is proved in \cite[Theorem 8.1]{GromovLawson1983Dirac} for $n=3$ and in \cite[Corollary 3]{ChodoshLiLiokumovich2023Classification} for $n=4,5$, so we only need to consider when $X$ is noncompact.

Suppose $Y = N \# X$ admits a complete metric with positive scalar curvature. Write $Y = N_\varepsilon \cup X_\varepsilon$ where $N_\varepsilon = N\setminus B$ and $X_\varepsilon = X\setminus B$ for some small balls. Then by Proposition \ref{prop: PSC deformation}, $Y$ admits an exhaustion by compact smooth codimension-zero submanifolds $K_1 \subset K_2 \subset \dots \subset Y$ such that for each $i\ge 1$, $K_i$ admits a metric $g_i$ such that $R_{g_i} > 0$ on $K_i$ and $H_{g_i}(\partial K_i) = 0$.

Take $i$ sufficiently large so that $N_\varepsilon \subset \Int K_i$. The double $DK_i$ admits a smooth metric of positive scalar curvature by \cite[Theorem 5.7]{GromovLawson1980Spin}. However, the connected component of $DK_i$ containing $N_\varepsilon$ is a connected sum of $N$ with another closed manifold, so it doesn't admit a metric of positive scalar curvature by \cite[Theorem 8.1]{GromovLawson1983Dirac}, \cite[Corollary 3]{ChodoshLiLiokumovich2023Classification}, a contradiction. 
\end{proof}

\section{Dimension 4}\label{sec: dim 4}

\subsection{Basics of Seiberg--Witten theory}
In this subsection we recall the basics of Seiberg--Witten theory. The exposition here mainly follows \cite[Chapters 1 and 2]{Nicolaescu2000SeibergWitten}, and we also refer to \cite[Chapters 1--4]{KronheimerMrowka2007Monopoles} and \cite[Sections 1--3]{HutchingsTaubes1999SeibergWitten}.

A $Spin^c$ structure on an oriented Riemannian $4$-manifold $(M,g)$ is a pair $\mathfrak{s} = (S,\bc)$, where the spinor bundle $S=S^+\oplus S^-$
is a $\mathbb Z_2$-graded Hermitian vector bundle with \(\operatorname{rank}_{\mathbb C}S^\pm=2\), and
$\bc \colon T^*M\longrightarrow\operatorname{End}_{\mathbb C}(S)$
is real-linear, odd, skew-adjoint Clifford multiplication satisfying
$$
\bc(\xi)\bc(\eta)+\bc(\eta)\bc(\xi)
=-2g^{-1}(\xi,\eta)\operatorname{id}_S.
$$
The Clifford multiplication naturally extends to $2$-forms by
$$
\bc(\xi\wedge\eta)=\tfrac12\bigl(\bc(\xi)\bc(\eta)-\bc(\eta)\bc(\xi)\bigr).
$$
The grading is chosen to agree with the orientation so that
$\bc(\alpha)|_{S^+}=0$ for every anti-self-dual $2$-form $\alpha$,
and $\bc(\beta)|_{S^-}=0$ for every self-dual $2$-form $\beta$.
An isomorphism of $Spin^{c}$ structures is a grading-preserving unitary bundle isomorphism over the identity of $M$ that intertwines Clifford multiplication.

We note that every oriented smooth Riemannian $4$-manifold $(M,g)$ admits a $Spin^c$ structure \cite{HirzebruchHopf1958Flaechenelemente,TeichnerVogt1994SpinC}. Here for simplicity we define a $Spin^c$ structure using a metric, but changing the metric gives a natural bijection between the isomorphism classes of $Spin^c$ structures. The set of isomorphism classes of  $Spin^{c}$ structures (when non-empty) is (noncanonically) in bijection with $H^2(M, \Z)$.

Let $\det(\mathfrak{s})$ be the determinant line bundle $\det(\mathfrak{s})=\det S^+\cong\det S^-$ of the $Spin^c$-structure $\mathfrak{s}$.

A unitary connection $A$ on $\det(\mathfrak{s})$, together with the Levi–Civita connection, induces a connection $\nabla^A$ on $S$ and a Dirac operator  $\slashed{D}_A \colon \Gamma(S^+)\longrightarrow\Gamma(S^-)$ given by 
$$
\slashed{D}_A\Phi=\sum_{j=1}^4\bc(e^j)\nabla^A_{e_j}\Phi,
$$
where $(e_j)_{j=1}^4$ is a local orthonormal frame and $(e^j)_{j=1}^4$ is its dual coframe.

Let $F_A$ denote the curvature 2-form $F_A \in i\Omega^2(M)$ on $\det(\mathfrak{s})$. It has self-dual part
$$
F_A^+=\tfrac12(F_A+*F_A)\in i\Omega^2_+(M).
$$

The Weitzenb\"ock formula for a positive spinor $\Phi\in\Gamma(S^+)$ is
\begin{equation}\label{eq: Weitzenbock}
\slashed{D}_A^{\,*}\slashed{D}_A\Phi
=(\nabla^A)^*\nabla^A\Phi
+\frac{R_g}{4}\Phi
+\frac12\bc(F_A^+)\Phi.   
\end{equation}

Denote by $\cA_{\mathfrak{s}}=\cA_{\mathfrak{s}}(M)$
the space of unitary connections on $\det(\mathfrak{s})$.
Also, denote by $c_{\mathfrak{s}}$ the first Chern class of
$\det(\mathfrak{s})$, i.e.,
$c_{\mathfrak{s}}=c_1(\det(\mathfrak{s}))$. By Chern--Weil theory,
$$(c_\fs)_{\mathbb R} = \left[\frac{i}{2\pi}F_A\right] \in H^2(M;\mathbb R).
$$

We can now define the configuration space
\[
\cC_{\mathfrak{s}}
=\cC_{\mathfrak{s}}(M)
=\Gamma(S^+)\times\cA_{\mathfrak{s}}.
\]

Fix a perturbation parameter, which is a real $2$-form $\eta \in \Omega^2(M)$. 
A \emph{$(\mathfrak{s},\eta)$-monopole} is a configuration
$C=(\Phi,A) \in \cC_{\mathfrak{s}}$ satisfying the Seiberg--Witten equations
\begin{equation}\label{eq: SW}
\begin{cases}
\slashed{D}_A\Phi=0,\\
\bc(F_A^+ + i\eta^+)=\dfrac12 q(\Phi).
\end{cases}
\end{equation}
Here the superscript ``$+$'' denotes the self-dual part of a
$2$-form, and $q \colon S^+ \to \End (S^+)$ is defined by
\begin{align}\label{eq: q}
q(\Phi) \coloneqq(\overline{\Phi}\otimes\Phi)_0
=\overline{\Phi}\otimes\Phi-\tfrac12|\Phi|^2\id_{S^+},   
\end{align}
where $\overline{\Phi}\otimes\Phi$ denotes the endomorphism $\Psi\mapsto\langle\Psi,\Phi\rangle\Phi$, and the subscript $0$ denotes the trace-free part. We take the Hermitian inner product to be linear in the first argument.

The energy of a configuration $C=(\Phi,A) \in \cC_{\mathfrak{s}}$ is defined as
\begin{align}\label{eq: energy}
E_\eta(\Phi,A)
= &\int_M \left(|\nabla^A\Phi|^2
+\frac{R_g}{4}|\Phi|^2
+\frac12\left|\frac12q(\Phi)-\bc(i\eta^+)\right|^2
+|F_A+2i\eta^+|^2
\right)\,dV_g,   
\end{align}
where for any endomorphism $T \colon S^+ \to S^+$ we have denoted $|T|^2\coloneqq  \tr(TT^*)$.

Here we record some useful identities for the quantities we introduced:
\begin{lemma}[{\cite[Lemma 2.1.5]{Nicolaescu2000SeibergWitten}}]\label{lem: identities}
Let $\alpha\in i\Omega^2_+(M)$,
$\Phi\in \Gamma(S^+)$, and
$T\in\operatorname{End}_0(S^+)$, where $\operatorname{End}_0$ denotes the trace-free endomorphisms.
Then we have the following pointwise identities:
\begin{align*}
|q(\Phi)|^2
&\coloneqq \tr\bigl(q(\Phi)^2\bigr)
= \frac{1}{2}|\Phi|^4,\\
|\bc(\alpha)|^2
&= 4|\alpha|^2,\\
\langle T,q(\Phi) \rangle
&\coloneqq \tr\bigl(Tq(\Phi)\bigr)
= \langle T\Phi,\Phi \rangle.
\end{align*}  
\end{lemma}

Using integration by parts, the Weitzenb\"ock formula \eqref{eq: Weitzenbock}, and the elementary identities in Lemma \ref{lem: identities}, we have the following identity for the energy:
\begin{lemma}[{\cite[Proposition 2.1.4]{Nicolaescu2000SeibergWitten}}]\label{lem: energy}
Let $(M,g)$ be a closed, connected, oriented 4-manifold.  Then 
\begin{align*}
E_\eta(\Phi,A)
={}&
\int_M
\left(
|\slashed{D}_A\Phi|^2
+\frac12\left|
\bc(F_A^+ + i\eta^+)-\frac{1}{2} q(\Phi)
\right|^2
\right)\,dV_g\\
&+4\int_M|\eta^+|^2\,dV_g
-4\pi^2\int_M c_{\fs}^2.
\end{align*}
\end{lemma}

\vspace{1em}

A $(\mathfrak{s},\eta)$-monopole is called \emph{reducible} if $\Phi\equiv 0$, and \emph{irreducible} otherwise.

The gauge group \(\cG_\mathfrak{s}=C^\infty(M,U(1))\) acts on $\cC_{\mathfrak{s}}$ by
$$
u\cdot(\Phi,A)=\bigl(u\Phi, \, A-2u^{-1}du\bigr).
$$

We denote by $\cZ_\fs = \cZ_\fs(g, \eta)$ the set of solutions of the Seiberg--Witten equations. $\cZ_\fs$ is a
$\cG_\fs$-invariant subset of $\cC_\fs$.
The \emph{Seiberg–Witten moduli space} is defined as $\fM_\fs = \fM_\fs(g, \eta) = \cZ_\fs/\cG_\fs$.

\vspace{1em}

Now further assume that the oriented 4-manifold $(M,g)$ is closed and connected. Assume $b_2^+(M)>1$ and choose a homological orientation, i.e., an orientation of $H^1(M;\R)\oplus H^2_+(M;\R)$.

Here we record a key estimate used to show the compactness property of the moduli space. It is proven using Kato's inequality and the Weitzenb\"ock formula.

\begin{lemma}[{\cite[Lemma 2.2.3]{Nicolaescu2000SeibergWitten}}]\label{lem: key estimate}
Suppose $(\Phi, A) \in \cZ_\fs$. Then on $M$,
$$\|\Phi\|^2_\infty \le 2 \max(- \min_M R + 4\|\eta^+\|_\infty, \, 0).$$
\end{lemma}

For generic perturbations $\eta$, the moduli space $\fM_\fs(g, \eta)$ is a smooth, compact, oriented manifold, and contains no reducibles. Its dimension is
$$d(\fs)= \frac{c_1(\det \fs)^2[M]-2\chi(M)-3\sigma(M)}4. $$

If \(d(\fs)=0\), define
$$\sw_M(\fs,g,\eta) =\#\fM_{\fs}(g,\eta)$$
as the signed count of points.

More generally,
$$ \sw_M(\fs,g,\eta) = 
\begin{cases} \left\langle\Omega_{\fs}^{k}, [\fM_{\fs}(g,\eta)]\right\rangle, &d(\fs)=2k \ge 0,\\ 
0,& d(\fs) < 0 \text{ or odd}. 
\end{cases} $$
Here $\Omega_\fs \in H^2(\fM_{\fs};\mathbb Z)$ is the first Chern class of the universal line bundle restricted to a fixed point of $M$ and $\Omega_\fs^k$ is the $k$-fold cup product.

Then we have
\begin{theorem}[{\cite[Theorem 2.3.5]{Nicolaescu2000SeibergWitten}}]
Suppose $M$ is a closed, connected, oriented and homology-oriented
smooth $4$-manifold such that $b_2^+(M)>1$. Then the correspondence
\[
Spin^c(M)\ni\mathfrak{s}
\mapsto
\sw_M(\mathfrak{s},g,\eta)
=:\sw_M(\mathfrak{s})\in\mathbb{Z}
\]
is independent of the metric $g$ and the generic perturbation $\eta$ and
is an invariant under diffeomorphisms preserving both the orientation and the homology orientation, with the $Spin^c$ structure transported by pullback.    
\end{theorem}

\subsection{Seiberg--Witten theory for cylindrical manifolds}

The Seiberg--Witten theory for cylindrical manifolds is well-developed, and the reader can refer to \cite[Chapter 4]{Nicolaescu2000SeibergWitten} and \cite{KronheimerMrowka2007Monopoles} for details.

Let $W$ be a compact $4$-manifold with boundary. Let $Y =\partial W$ be its boundary.
Let $W^*$ be the noncompact Riemannian manifold obtained by attaching a copy of the cylinder $Y\times [0,\infty)$, identifying $\partial W$ with $Y \times \{0\}$: $W^* = W \cup_Y (Y\times [0,\infty))$. In the following we will consider $(W, g)$ with product metric $g_W = g_Y + dt^2$ near the boundary, and get the cylindrical completion $(W^*, g^*)$ by extending the product metric to the cylinder $Y\times [0,\infty)$. Given a $Spin^c$ structure $\fs_W$ on $(W, g)$ with product metric $g_W = g_Y + dt^2$ near the boundary, we use a product identification for the $Spin^c$ structure, and extend the product so we get a $Spin^c$ structure $\fs_{W^*}$ on $W^*$.

We consider the Seiberg--Witten equations \eqref{eq: SW} on the cylindrical completion $W^*$.
In this case the perturbation parameters will be compactly supported real 2-forms
$\eta \in \Omega^2_c(W^*)$, and we will consider finite-energy monopoles $(\Phi, A)$ where $E_\eta(\Phi, A) < \infty$.

\begin{lemma}\label{lem: torsion}
Let $W^* = W \cup_Y (Y\times [0,\infty))$ be a cylindrical $4$-manifold with a Riemannian metric $g_{W^*}$ that restricts to the product metric $g_Y + dt^2$ on the cylindrical end. Let $\fs_{W^*}$ be a $Spin^c$ structure on $W^*$ and let $L \coloneqq \det(\fs_{W^*})$ be its determinant line bundle. If $A$ is a unitary connection on $L$ whose curvature $F_A$ is square-integrable on $W^*$, then, for every $t\ge 0$, $c_1 (L|_{Y \times \{t\}})\in H^2(Y;\mathbb Z)$ is a torsion class.
\end{lemma}
\begin{proof}
Let $L_t=L|_{Y \times \{t\}}$ for each $t\ge 0$. All $L_t$ are topologically isomorphic, so $c:=c_1(L_t) \in H^2(Y;\mathbb Z)$ is independent of $t$. Write $\beta_t:=F_A|_{Y\times\{t\}}$, then $F_A=\beta_t+dt\wedge \alpha_t$ where
$\alpha_t\in i\Omega^1(Y)$.
The 2-form $\beta_t$ is closed and $\left[\frac{i}{2\pi}\beta_t\right]=c_{\mathbb R}\in H^2(Y;\mathbb R)$. If $c$ is not torsion, then $c_\R \neq 0$, so we can find a nonzero harmonic representative $h$ of $c_\R$ by Hodge theory. Then $\|\beta_t\|_{L^2(Y)}^2
\ge \|2\pi \,h\|_{L^2(Y)}^2 =:C>0$. Then integration yields 
$$\int_{Y\times[0,\infty)}|F_A|^2 dV = \int_0^\infty\int_Y
\bigl(|\beta_t|^2+|\alpha_t|^2\bigr)\,dV_Y\,dt
\ge
\int_0^\infty\|\beta_t\|_{L^2(Y)}^2\,dt
\ge
\int_0^\infty C\,dt
=\infty,$$
contradicting $F_A\in L^2$.    
\end{proof}

On the one hand, analogous to the closed case, uniformly positive scalar curvature together with $b_2^+(W)>0$ allows us to get a perturbation with no solutions:
\begin{proposition}
Let $W$ be a compact connected oriented $4$-manifold with $b_2^+(W)>0$, carrying a metric of positive scalar curvature that is a product near its boundary. Let $W^*$ be its cylindrical completion. For every fixed $Spin^c$ structure, there exists a smooth, compactly supported real 2-form $\eta$ such that no finite-energy $(\fs_{W^*}, \eta)$-monopoles exist on $W^*$.   
\end{proposition}
\begin{proof}
By our assumption there exists $R_0 > 0$ such that $R \ge R_0$ on $W^*$. Let $\eta \in \Omega^2_c(W^*)$ be an arbitrary compactly supported real 2-form such that $\|\eta\|_{L^\infty} \le \frac{R_0}{8}$. First we show that uniformly positive scalar curvature excludes irreducible solutions through the Weitzenb\"ock formula. 

Let $(\Phi, A)$ be a finite-energy $(\fs_{W^*}, \eta)$-monopole on $W^*$. It follows from \eqref{eq: energy}, $R \ge R_0>0 $ and $\eta \in \Omega^2_c(W^*)$ that $E_\eta(\Phi, A) < \infty$ implies $\int_{W^*}|\Phi|^2< \infty$ and $\int_{W^*} |F_A|^2 < \infty$.

Let $\chi \colon \R \to [0,1]$ be a smooth cutoff function such that $\chi(t) \equiv 1$ for $t \le 0$, and $\chi(t) \equiv 0$ for $t \ge 1$.

For each $L \ge 0$, let $\chi_L \colon W^* \to [0,1]$ be the cutoff function such that $\chi_L (t, y) = \chi(t-L)$ on $[L, L+1] \times Y$, $\chi_L \equiv 1$ on $W\cup(Y\times[0,L])$, and $\chi_L \equiv 0$ on $Y\times[L+1,\infty)$. Thus there exists $C >0$ independent of $L$ such that $|d\chi_L| \le C$.

By the Weitzenb\"ock formula \eqref{eq: Weitzenbock}, the Seiberg--Witten equations \eqref{eq: SW}, and \eqref{eq: q},
\begin{align*}
0 & = \slashed{D}_A^{\,*}\slashed{D}_A\Phi \\
& =(\nabla^A)^*\nabla^A\Phi
+\frac{R_g}{4}\Phi
+\frac12\bc(F_A^+)\Phi \\
& =(\nabla^A)^*\nabla^A\Phi
+\frac{R_g}{4}\Phi
+\frac12( \frac12q(\Phi) - \bc(i \eta^+))\Phi \\
& =(\nabla^A)^*\nabla^A\Phi
+\frac{R_g}{4}\Phi
+ \frac18 |\Phi|^2\Phi 
- \frac12\bc(i\eta^+)\Phi.
\end{align*}

Taking the real part of the $L^2$ inner product with $\chi_L^2 \Phi$ gives

\begin{align*}
0   & =   \Re \int_{W^*}   \langle \big((\nabla^A)^*\nabla^A\Phi
+\frac{R_g}{4}\Phi
+ \frac18 |\Phi|^2\Phi 
- \frac12\bc(i\eta^+)\Phi \big), \,  \chi_L^2 \Phi \rangle dV_g\\
& =  \Re \int_{W^*}  \chi_L^2|\nabla^A\Phi|^2 + 2  \chi_L \langle \nabla^A \Phi, d \chi_L \otimes \Phi\rangle 
+\frac{R_g}{4}\chi_L^2|\Phi|^2 + \frac{1}{8} \chi_L^2 |\Phi|^4 
 \\
  & \quad - \frac12 \chi_L^2\langle \bc(i\eta^+)\Phi, \Phi \rangle \\
  & \ge  \int_{W^*} - |d\chi_L|^2|\Phi|^2
+\frac{R_0}{4}\chi_L^2|\Phi|^2 + \frac{1}{8} \chi_L^2 |\Phi|^4 - \frac{R_0}{8} \chi_L^2 |\Phi|^2 \\
& \ge   - C\int_{[L, L+1] \times Y}|\Phi|^2
+\frac{R_0}{8} \int_{W^*}\chi_L^2|\Phi|^2   \\
 & \xrightarrow{L\to \infty} \frac{R_0}{8} \int_{W^*}|\Phi|^2.
\end{align*}
This shows for a solution $(\Phi, A)$, we must have $\Phi \equiv 0$.

Thus all solutions to \eqref{eq: SW} are reducible. Then \cite[Proposition 4.4.1]{Nicolaescu2000SeibergWitten} and \cite[Corollary 27.2.3]{KronheimerMrowka2007Monopoles}  further show that using $b_2^+(W) > 0$, we can choose a sufficiently small perturbation $\eta$ such that no solutions exist. Here for completeness we briefly sketch the argument. Since all solutions to \eqref{eq: SW} are reducible, we must have $F_A^+ = - i\eta^+$. By cylindrical Hodge theory, since $b_2^+(W) > 0$, there exists a nonzero square-integrable closed self-dual real 2-form $\kappa$ on $W^*$. By Lemma \ref{lem: torsion}, $F_A \in L^2$ implies that $c_1(\det \fs_{W^*})$ restricts to a torsion class on every cross-section of the end cylinder. Then by \cite[Lemma 27.2.1]{KronheimerMrowka2007Monopoles} we must have 
$$\int_{W^*} - i\eta^+ \wedge \kappa = \int_{W^*} F_A^+ \wedge \kappa = -2\pi i \langle c_1(\det \fs_{W^*}), [\kappa] \rangle.
$$
Here the pairing on the right-hand side is computed using a compactly supported closed representative of the real Chern class, and is independent of $\eta$.
We choose $L$ so that $\int_{W^*}\chi_L|\kappa|^2>0$.
Taking $\eta=\varepsilon\chi_L\kappa$, there is a unique $\varepsilon$ for which the above identity holds. We thus choose $\varepsilon$ sufficiently small so that $\|\eta\|_{L^\infty} \le \frac{R_0}{8}$, and $\varepsilon$ avoids that unique value. Then no solutions to the perturbed Seiberg--Witten equations \eqref{eq: SW} exist with this choice of $\eta$.
\end{proof}

On the other hand, using local compactness, we can stretch the neck to get a solution. This neck-stretching argument is inspired by the proof strategy of \cite[Theorem 1.3]{Kato2021HarmonicTheory}.

\begin{proposition}\label{prop: neck}
Let $M$ be a closed connected oriented smooth four-manifold with $b_2^+(M)>1$, and suppose $\sw_M(\mathfrak s)\neq 0$ for some $Spin^{c}$ structure $\fs$. Suppose $W \subset M$ is a compact codimension-zero submanifold with smooth boundary $Y$, and $W$ carries a metric of positive scalar curvature $g$ that is a product $g_Y + dt^2$ near its boundary. Let $(W^*, g^*)$ be its cylindrical completion and denote the product extension of $\fs|_W$ to $W^*$ by $\fs_{W^*}$. Then for any compactly supported real 2-form $\eta \in \Omega^2_c(W^*)$, there exists a smooth finite energy $(\fs_{W^*}, \eta)$-monopole on $W^*$.
\end{proposition}
\begin{proof}
Write $W^*= W \cup_Y (Y \times [0,\infty))$. Enlarging the compact core $W$ by a finite cylindrical
segment and retaining the notation $W$, we may assume $\eta \in  \Omega^2_c(W)$.

Let $V\coloneqq  \overline{M \setminus W}$. Let $h$ be any metric on $V$ such that $h=dt^2+g_Y$ on a boundary collar.

For each $T \in \N$, write $W_T = W \cup_{Y \times \{0\}} (Y \times [0,T])$. Identify $M$ with $W_T \cup_{Y} V$. Let $g_T$ be the metric on $M$ such that 
$$g_T =\begin{cases}
g &\text{on $W$}\\
g_Y + dt^2 &\text{on $Y \times [0,T]$}\\
h &\text{on $V$.}\\
\end{cases}
$$
Then $g_T$ has positive scalar curvature on $W_T$. Notice that the $Spin^c$ structure $\fs$ of $M$ remains the same under these metrics. 

Since $\sw_M(\mathfrak s)\neq 0$, on each $(M, g_T)$, the Seiberg--Witten equations admit a solution for generic perturbations. Approximating $\eta$ by generic perturbations and applying Seiberg--Witten compactness on $(M,g_T)$ shows that $\sw_M(\fs)\neq 0$ guarantees the existence of a $(\fs,\eta)$-monopole $(\Phi_T,A_T)$.

By Lemma \ref{lem: energy}, 
$$E_{\eta}(\Phi_T, A_T) = 4\int_M|\eta^+|^2\,dV_{g_T}
-4\pi^2\int_M c_{\fs}^2,$$
so $E_{\eta}(\Phi_T, A_T)$ is uniformly bounded on $M$. 

Using Lemma \ref{lem: key estimate} to bound $\Phi_T$ in terms of $R$ and $\eta$, we see that restricting to $W_T$, we have
\begin{align*}
E_{\eta}(\Phi_T|_{W_T}, A_T|_{W_T}) & \le   E_{\eta}(\Phi_T, A_T) + \int_V \frac{\max(0, - R_h)}{4}|\Phi_T|^2 
dV_h \le C
\end{align*}
for some uniform constant $C$ depending on $(V,h)$ and $\eta$ but independent of $T$.

Then using Lemma \ref{lem: identities} and the elementary norm inequality $|a+b|^2\le2|a|^2+2|b|^2$, this in turn gives uniform upper bounds on $\|F_{A_T}\|_{L^2(W_T)}$, $\|\Phi_{T}\|_{L^4(W_T)}$ and $\|\nabla^{A_T} \Phi_{T}\|_{L^2(W_T)}$.

Since $E_{\eta}(\Phi_T, A_T)$, $\|F_{A_T}\|_{L^2(W_T)}$, $\|\Phi_T\|_{L^4(W_T)}$ and $\|\nabla^{A_T} \Phi_T\|_{L^2(W_T)}$ are uniformly bounded, we can apply the local compactness theory (see \cite[Theorem 5.2.1]{KronheimerMrowka2007Monopoles} and also \cite[Section 4.4.2]{Nicolaescu2000SeibergWitten}) to conclude that there is a sequence of smooth gauge transformations $u_T: W_T \to S^1$ such that after passing to a subsequence, the transformed configurations $u_T(\Phi_T, A_T)$ converge weakly to a configuration $(\Phi_\infty, A_\infty)$ on $W^*$ that solves \eqref{eq: SW}, and the convergence is strong in $C^\infty$ on every compact subset of $W^*$.

From the weak convergence we conclude that $E_\eta(\Phi_\infty, A_\infty)$ is bounded. Thus $(\Phi_\infty, A_\infty)$ is a desired $(\fs_{W^*}, \eta)$-monopole on $W^*$.
\end{proof}

Combining the previous two results, we thus get
\begin{proposition}\label{prop: obstruction}
Let $M$ be a closed, connected, oriented, smooth four-manifold with $b_2^+(M)>1$, and suppose $\sw_M(\mathfrak s)\neq 0$ for some $Spin^{c}$ structure $\mathfrak s$.
If $W\subset M$ is a compact connected smooth domain with $b_2^+(W)>0$, then $W$ cannot carry a metric of positive scalar curvature that is a product near the boundary.   
\end{proposition}

\subsection{Proof of Theorem \ref{thm: dim 4}}

\begin{lemma}\label{lem: b_2^+}
Let $X$ be a compact connected oriented smooth $4$-manifold with boundary and suppose $W \subset X$ is a smooth connected codimension-zero submanifold such that $X\setminus W$ is contained in a collar neighborhood of $\partial X$. Then $b_2^+(W) = b_2^+(X)$.
\end{lemma}
\begin{proof}
By removing part of the collar, we can get $X' \subset X$ diffeomorphic to $X$ with $X' \xhookrightarrow{\iota_1} W \xhookrightarrow{\iota_2} X$.

By Poincar\'e duality, $H_2(X';\R) \cong H^2_c(\Int X'; \R)$.
We consider the intersection form extended to $\R$ coefficients $Q_{X'} \colon H^2_c(\Int X';\R) \times H^2_c(\Int X';\R) \to \R$ given by $Q_{X'}(a,b) = \int_{X'} a \wedge b$. Let $(\iota_1)_!$ denote extension by zero.
We compute
$$
Q_{X'}(a,b) = \int_{X'} a \wedge b  = \int_{W} (\iota_1)_!(a) \wedge (\iota_1)_!(b) = Q_W((\iota_1)_!(a), (\iota_1)_!(b)).
$$
Hence the image of any positive definite subspace of $H^2_c(\Int X';\R)$ remains positive in $H^2_c(\Int W;\R)$. Thus $b_2^+ (X') \le b_2^+(W)$.

Similarly we can show $b_2^+ (W) \le b_2^+(X)$. Since $X'$ and $X$ are diffeomorphic, $b_2^+ (X') = b_2^+ (X)$. We thus have $b_2^+(W) = b_2^+ (X)$ as desired.
\end{proof}

\begin{lemma}\label{lem: puncture}
Let $M$ be a closed connected oriented smooth $4$-manifold. Let $B_1, \dots, B_k$ be pairwise disjoint open balls in $M$. Let $U=\cup_{i=1}^k B_i$ and let $X = M \setminus B$. Then $b_2^+(M) = b_2^+(X)$.
\end{lemma}
\begin{proof}
Since $H^1(U;\mathbb Z)=H^2(U;\mathbb Z)=0$, the long exact sequence of the pair $(M,U)$ gives an isomorphism
$H^2(M,U;\mathbb Z)\xrightarrow{\cong} H^2(M;\mathbb Z)$.

By excision,
$H^2(M,U;\mathbb Z)\cong H_c^2(M\setminus B;\mathbb Z)
=H_c^2(X;\mathbb Z)$. Hence
$H_c^2(X;\mathbb Z)\cong H^2(M;\mathbb Z)$.

This isomorphism identifies the integral intersection forms. Thus their positive indices agree, so $b_2^+(X)=b_2^+(M)$.
\end{proof}

\begin{proof}[Proof of Theorem \ref{thm: dim 4}]
Suppose to the contrary that $\Int X$ admits a complete metric of uniformly positive scalar curvature. Then by Theorem \ref{thm: Wang--Zhu}, $\Int X$ admits an exhaustion by compact smooth codimension-zero submanifolds $K_1 \subset K_2 \subset \dots \subset \Int X$ and a complete metric $g$ of uniformly positive scalar curvature such that near each hypersurface $\Sigma_i \coloneqq  \partial K_i$, the metric is of product form $g = dt^2 + g_{\Sigma_i}$.

For large enough $i$, $\Int X \setminus K_i$ lies in a cylinder diffeomorphic to $\partial X \times (0, \infty)$. Fix such an $i$. Then $\Int X \setminus K_i$ lies in a collar neighborhood of $\partial X$, so Lemma \ref{lem: b_2^+} gives $b_2^+(K_i) = b_2^+(X) >0$. Then by Proposition \ref{prop: obstruction}, $K_i$ cannot carry a metric of positive scalar curvature that is a product near the boundary, a contradiction.
\end{proof}

\section{Dimensions 5, 6, 7}\label{sec: dim 5,6,7}
\subsection{Bordism and tangential structures}
We first present some basic facts about bordism groups. The reader can consult \cite[Chapter 8]{DavisKirk2001LectureNotes} for more details. 
We let $\Omega_n^{SO}$, respectively $\Omega_n^{Spin}$, denote the bordism group of $n$-dimensional closed manifolds equipped with an orientation, respectively a spin structure. Addition is given by disjoint union and inverse is given by reversing the orientation. More generally, for a topological space $Z$ and $G=SO$ or $G=Spin$, let $\Omega_n^G(Z)$ denote the bordism group of pairs $(M,f)$, where $M$ is an oriented, respectively spin, closed $n$-manifold and
$f \colon M \to Z$ is a continuous map. Under this notation, $\Omega_n^{G} = \Omega_n^G (*)$. Let $\Omega_n^{G,+}(Z)\subset \Omega_n^G(Z)$ be the subgroup consisting of those bordism classes which admit a representative $(M,f)$ such that $M$ carries a metric of positive scalar curvature. 

The bordism groups $\Omega_n^G(Z)$ form a generalized homology theory satisfying additivity \cite[Theorem 8.26 and Theorem 8.32]{DavisKirk2001LectureNotes}. For a based space \(Z\), define the reduced bordism group
$$
\widetilde\Omega^{G}_n(Z)
=
\ker \left(\Omega^{G}_n(Z)\longrightarrow\Omega^{G}_n(*)\right).
$$
The map here forgets the map to \(Z\). The basepoint inclusion splits it, giving
\begin{align}\label{eq: reduced bordism}
\Omega^{G}_n(Z) \cong \Omega^{G}_n (*) \oplus\widetilde\Omega^{G}_n(Z).
\end{align}
Then the reduced bordism groups satisfy the homotopy, exactness, suspension and additivity axioms. In particular, the suspension axiom gives the isomorphism
\begin{align}\label{eq: suspension}
\widetilde\Omega_n^{G}(Z)\cong
\widetilde\Omega_{n+1}^{G}(\Sigma Z) 
\end{align}
for all $n$, where $\Sigma Z = (Z\times [0,1])/(Z\times \{0\} \cup Z\times \{1\} \cup \{x_0\} \times [0,1])$ is the reduced suspension. The additivity axiom gives the following isomorphism for a wedge sum $Z =\bigvee_{j \in J} Z_j$
\begin{align}\label{eq: additivity}
\bigoplus_{j \in J} \widetilde\Omega_n^{G}(Z_j) \cong \widetilde\Omega_n^{G}(Z)
\end{align}
for all $n$.

We will use these bordism groups when $n$ is 4 or 5. 
It is well-known that $\Omega_4^{SO} \cong \Z$; the complex projective space $\CP^2$ is a generator, and the signature $\sigma(M)\coloneqq  b_2^+(M) - b_2^-(M)$ of a closed oriented 4-manifold $M$ defines an isomorphism $\Omega_4^{SO} \to \Z$ (see e.g. \cite[p. 226]{DavisKirk2001LectureNotes}). In particular, a closed oriented 4-manifold bounds a compact oriented 5-manifold if and only if the signature is zero. 

Since $S^1$ is the suspension of $S^0$ and $\Omega_4^{SO}(S^0) \cong \Omega_4^{SO}(*) \oplus \Omega_4^{SO}(*)$, using \eqref{eq: suspension} and \eqref{eq: reduced bordism},
\begin{align}\label{eq: bordism S^1}
\widetilde\Omega^{SO}_5(S^1)
\cong
\widetilde\Omega^{SO}_4(S^0)
\cong
\Omega^{SO}_4 \cong \Z.
\end{align}
Geometrically, take a class $
[N^5,f]\in \widetilde\Omega^{SO}_5(S^1)$ such that $f \colon N\to S^1$ is a smooth representative of its homotopy class. Choose a regular value $p\in S^1$. Then $M=f^{-1}(p)$
is a closed oriented $4$-manifold. The suspension boundary map sends $[N,f]\mapsto [M]\in \Omega^{SO}_4$.
Thus the isomorphism is simply
$$[N,f]\mapsto [f^{-1}(p)].
$$
Since
$\Omega^{SO}_4\cong\mathbb Z$
via the signature, the class of $(N,f)$ is determined by $\sigma(f^{-1}(p))$.

Conversely, given a closed oriented $4$-manifold $M$, we can take $[M\times S^1, \operatorname{pr}_{S^1}]$ as a representative of the corresponding class in $\widetilde\Omega^{SO}_5(S^1)$. Since $\CP^2$ generates $\Omega^{SO}_4$, we find that $[\CP^2 \times S^1, \operatorname{pr}_{S^1}]$ generates the group $\widetilde\Omega^{SO}_5(S^1)$.

\vspace{1em}

We will need the following bordism theorem due to Gromov--Lawson \cite[Theorems B and C]{GromovLawson1980Classification} and Rosenberg \cite[Theorems 2.2 and 2.13]{Rosenberg1986CStarII}:
\begin{theorem}[see e.g. {\cite[Theorem 1.3]{Stolz2023Positive}}]\label{thm: psc bordism}
Let $M$ be a closed connected $n$-manifold, $n\ge 5$, with fundamental group $\pi$, and let
\[
u\colon M\to B\pi
\]
be the classifying map of the universal covering $\widetilde M\to M$. That is, $u$ induces an isomorphism on the fundamental group.
\begin{enumerate}
\item If $M$ is spin, then $M$ admits a metric of positive scalar curvature if and only if
\[
[M,u]\in \Omega_n^{Spin,+}(B\pi).
\]

\item If $M$ is oriented and $\widetilde M$ is nonspin, then $M$ admits a metric of positive scalar curvature if and only if
\[
[M,u]\in \Omega_n^{SO,+}(B\pi).
\]
\end{enumerate}    
\end{theorem}

We next recall the language of tangential
structures used in the formulation of surgery principles for positive scalar curvature and boundary mean curvature; see e.g. \cite{CecchiniRadeZeidler2023Nonnegative, Frenck2023Nonspin,Frenck2026Doubling}.

Let $BO(n)$ denote the classifying space for real vector bundles of rank $n$. We call the map $f \colon X \to BO(n)$ corresponding to a vector bundle $V \to X$ the classifying map for $V$. The classifying map for $V \oplus \R \to X$ is given by the composition of the classifying map $f \colon X \to BO(n)$ with the natural map
$BO(n) \to BO(n + 1)$.

The map $X \to BO(n) \to BO \coloneqq \colim_{n\to \infty} BO(n)$ into the colimit of these natural maps is called the stable classifying map for $V$. The space $BO$ classifies stable vector bundles. A map $f \colon X \to BO$ from a compact space lands in a finite stage $BO(n)$ of this colimit and
hence corresponds to (the stable isomorphism class of) a vector bundle.

\begin{definition}[tangential structure]
A \emph{(stable) tangential structure} is a map $\theta: B \to BO.$
Let $V\to X$ be a vector bundle with stable classifying map $f \colon X \to BO$. A $\theta$-structure on $V$ is a lift $\ell \colon X \to B$ of $f$ through $\theta$, i.e. $f=\theta\circ\ell$. 
A $\theta$-structure on a manifold $M$ is a $\theta$-structure on the tangent bundle $TM$.
\end{definition}
Standard examples for tangential structures are $B = BSO$ (orientations) and $B = BSpin$ (spin structures).

\begin{definition}[$\theta$-cobordism]
Let $\theta\colon B\to BO$ be a tangential structure.
A $\theta$-cobordism between closed $n$-manifolds $M_0$ and $M_1$ equipped with $\theta$-structures is a compact $(n+1)$-manifold $W$ with boundary $\partial W=M_0^{\op}\sqcup M_1$, together with a $\theta$-structure on its tangent bundle whose induced boundary structures agree with the prescribed structures on $M_0$ and $M_1$. Here the boundary structures are induced using $TW|_{\partial W}\cong T(\partial W)\oplus\mathbb R$.

If $M_0$ and $M_1$ have identified boundaries $Y$, a $\theta$-cobordism \emph{relative to the boundary} is a cobordism with corners whose boundary faces are $M_0$, $M_1$, and $Y\times[0,1]$, with the usual identifications along their boundaries, and whose $\theta$-structure on the vertical face $Y\times[0,1]$ is the product of the common boundary structure.
\end{definition}
\begin{example}
Let $Z$ be a topological space.
Let $B = BSO \times Z$, and let $\theta \colon B \to BO$ be the composition of the projection and the natural map $BSO \times Z \to BSO \to BO$. Then two closed oriented manifolds equipped with maps to $Z$ are $\theta$-cobordant if and only if they represent the same element of the bordism group $\Omega_n^{SO}(Z)$.

Similarly, if we let $B' = BSpin \times Z$, and let $\theta' \colon B' \to BO$ be the composition of the projection and the natural map $BSpin \times Z \to BSpin \to BO$, then two closed spin manifolds  equipped with maps to $Z$  are $\theta'$-cobordant if and only if they represent the same element of the bordism group $\Omega_n^{Spin}(Z)$.
\end{example}

\begin{definition}[tangential 2-type]
Let $M$ be a manifold. A tangential structure $\theta \colon B \to BO$ is called the \emph{(stable) tangential $2$-type} of $M$ if the tangent bundle $TM$ admits a $2$-connected $\theta$-structure $\ell_M \colon M \to B$ and $\theta$ is $2$-coconnected.

Here, $2$-connected means that $\ell_M$ induces isomorphisms on $\pi_0$ and $\pi_1$ and a surjection on $\pi_2$, while $2$-coconnected means that $\theta$ induces an injection on $\pi_2$ and isomorphisms on $\pi_k$ for all $k\ge 3$.
\end{definition}

Tangential 2-types always exist, as they are given by the second stage of the Moore--Postnikov tower for the classifying map $\tau \colon M \to BO$ of the stable tangent bundle $TM$ of $M$.

\begin{example}
Let $M$ be a connected, orientable manifold of dimension $d \geq 2$.
\begin{enumerate}
\item If $M$ is spin, take $B = B Spin \times B\pi_1(M)$. Then its tangential $2$-type $\theta \colon B \to BO$ is given by the composition
    of the projection and the natural map
    \[
    B Spin \times B\pi_1(M) \to B Spin \to BO.
    \]

\item If $M$ is totally nonspin, take $B = BSO \times B\pi_1(M)$. Then its tangential $2$-type $\theta \colon B \to BO$ is given by the composition of the projection and the natural map
    \[
    BSO \times B\pi_1(M) \to BSO \to BO.
    \]
\end{enumerate}    
\end{example}

\begin{definition}\label{def: extension}
Let $\Sigma\subset M$ be a hypersurface, and let $\theta\colon B \to BO$ be the tangential $2$-type of $\Sigma$. We say that the tangential $2$-type of $\Sigma$ \emph{extends to $M$} if there exist $\theta$-structures $\ell_\Sigma \colon \Sigma \to B$, 
$\ell_M \colon M \to B$
such that 
\begin{enumerate}
\item $\ell_\Sigma$ is $2$-connected,
\item $\ell_\Sigma$ is homotopic to $\ell_M\circ\iota$, where $\iota: \Sigma\hookrightarrow M$ is the inclusion.
\end{enumerate}
\end{definition}

\begin{example}\label{ex: product tangential structure}
Let $\theta\colon B\to BO$ be the tangential $2$-type of a closed manifold $\Sigma$, and let $\ell_\Sigma\colon \Sigma\to B$ be a $2$-connected $\theta$-structure. Then the tangential $2$-type of $\Sigma$ extends to the cylinder $C=\Sigma\times[0,1]$ by setting
$\ell_C=\ell_\Sigma\circ \proj_\Sigma \colon C \to B.$   

Consequently, if $W\subset\Sigma\times[0,\infty)$ is a compact smooth codimension-zero submanifold containing a neighborhood of $\Sigma\times\{0\}$, then $\ell_W=\ell_\Sigma\circ\proj_\Sigma|_W$ defines a $\theta$-structure on $W$ extending the given structure on $\Sigma\times\{0\}$.
\end{example}

\subsection{Proof of Theorem \ref{thm: UPSC dim 6, 7} and Theorem \ref{thm: PSC dim 5, 6, 7}}
\label{subsec: dim 5,6,7}

We present a lemma analogous to \cite[Proposition 6.4]{Rade2023MuBubbles}.
\begin{lemma}\label{lem: psc product}
Let $Y$ be a closed connected oriented manifold of dimension $n-1$. Let $W\subset Y\times[0,\infty)$ be a compact connected smooth codimension-zero submanifold containing a neighborhood of $Y\times\{0\}$, with
\[
\partial W=(Y\times\{0\})\sqcup\Sigma,
\qquad
\Sigma\subset Y\times(0,\infty).
\]
Here $\Sigma$ need not be connected.

\begin{enumerate}
\item \label{item:n>=6} Let $n \ge 6$. Suppose $W$ admits a metric $g$ of positive scalar curvature that is a product near $\Sigma$.
Then $Y\times[0,1]$ admits a metric $\widetilde g$ of positive scalar curvature agreeing with $g$ near $Y\times\{0\}$ and equal to a product near $Y\times\{1\}$.
\item \label{item:n>=5} Let $n \ge 5$. Suppose $W$ admits a metric $g$ of positive scalar curvature for which $\Sigma$ is minimal.
Then $Y\times[0,1]$ admits a metric $\widetilde g$ of positive scalar curvature agreeing with $g$ near $Y\times\{0\}$ and for which $Y\times\{1\}$ is minimal.
\end{enumerate}

\end{lemma}
\begin{proof}
In both cases, regard $W$ as a cobordism
$W\colon Y\rightsquigarrow\Sigma$.
Restricting the projection $Y\times[0,\infty)\to Y$ gives a retraction $r\colon W\to Y$.
Let $\theta \colon B \to BO$ be the tangential 2-type of $Y$. Then $\theta$ extends to $W$ by Example \ref{ex: product tangential structure}. Choose the $2$-connected $\theta$-structure
$\ell_Y\colon Y\to B$ defining the tangential $2$-type of $Y$, and write $f_Y=\theta\circ\ell_Y$.
Set $\ell_W=\ell_Y\circ r$ and $f_W=f_Y\circ r$. Let $\nu(Y)$ denote a stable normal bundle of $Y$, with a chosen stable trivialization of $TY\oplus\nu(Y)$. Since $W\subset Y\times[0,\infty)$, we have $TW\cong r^*TY\oplus\mathbb R$, so pulling back the chosen
trivialization to $W$ gives a stable trivialization of $TW\oplus r^*\nu(Y)
\cong r^*(TY\oplus\nu(Y))\oplus\mathbb R$.
We use the $\theta$-structure induced by this identification.

We first proceed as in
\cite[Proposition 6.4]{Rade2023MuBubbles} and perform
interior surgeries preserving the tangential structure. Note that the tangential data are already implicitly preserved through the surgery traces in \cite[Proposition 6.4]{Rade2023MuBubbles}; here we just make this explicit. Since $\ell_Y=\ell_W\circ\iota\colon Y\to B$ is $2$-connected, \((\ell_Y)_*\) is an isomorphism on $\pi_1$ and a surjection on $\pi_2$, and $(\ell_W)_*$ is surjective on both $\pi_1$ and $\pi_2$. 

For each homotopy $H\colon V\times[0,1]\to Y$ used below, the map $f_Y\circ H(\,\cdot\,,0)$ classifies the stable tangent bundle of $V$. Hence $f_Y\circ H$ classifies the stable tangent bundle of $V\times[0,1]$, with strict lift $\ell_Y\circ H$.
Thus in the steps below, whenever we homotope a retraction, we include the corresponding cylinder before the surgery trace. All such homotopies are fixed near the boundary, so these cylinders are $\theta$-cobordisms relative to the boundary.

\emph{Step 1:} We kill $K_1 \coloneqq \ker\bigl(r_*\colon\pi_1(W)\to\pi_1(Y)\bigr) = \ker\bigl((\ell_W)_*\colon\pi_1(W)\to\pi_1(B)\bigr)$. This kernel is finitely generated as a normal subgroup of $\pi_1(W)$, since $\pi_1(W)$ is finitely generated and $\pi_1(Y)$ is finitely presented. Thus we can choose a finite number of normal generators.

Let $\alpha$ be one of the chosen normal generators of $K_1$. Since $1 < n/2$ we can represent $\alpha$ by an embedding $\alpha \colon S^1 \hookrightarrow \Int W$. Because $r\circ\alpha$ is null-homotopic in $Y$, we can choose an extension 
\[
F_\alpha\colon D^2\to Y,
\qquad
F_\alpha|_{S^1}=r\circ\alpha.
\]

Since $W$ is oriented the normal bundle of $\alpha$ is trivial. Using
the chosen null-homotopy $F_\alpha$, we can thus choose a framing of $\nu(\alpha, W)$ that is compatible with the stable trivialization of $r^*\nu(Y)\oplus TW$, and choose a corresponding tubular neighborhood $S^1\times D^{n-1}$. We homotope $r$, fixing the central loop and a neighborhood of $\partial W$, so that it is constant in the normal-disk direction on this neighborhood. We transport the stable trivialization through this homotopy and continue to set $\ell_W=\ell_Y\circ r$. 

The surgery trace $T_\alpha\colon W\rightsquigarrow W_\alpha$ is obtained by attaching the $2$-handle $D^2\times D^{n-1}$ to $W\times[0,1]$.
Extend $r$ to a map $r_{T_\alpha}\colon T_\alpha\to Y$ by $r_{T_\alpha} = r \circ \proj_W$ on $W \times [0,1]$, and $r_{T_\alpha}(z,v)=F_\alpha(z)$ on the handle. Set $\ell_{T_\alpha}=\ell_Y\circ r_{T_\alpha}$ and $f_{T_\alpha}=f_Y\circ r_{T_\alpha}$. By the choice of framing, the stable trivialization extends to $r_{T_\alpha}^*\nu(Y)\oplus TT_\alpha$. By our construction, $f_{T_\alpha}$ is a stable tangent classifying map extending $f_W$, and $\theta\circ\ell_{T_\alpha}=f_{T_\alpha}$. Thus $T_\alpha$ is a $\theta$-cobordism.

Repeating this construction for all the normal generators gives a $\theta$-cobordism relative to the boundary $T_1\colon W\rightsquigarrow W_1$ and a retraction $r_{T_1}\colon T_1\to Y$. Let $r_1 \coloneqq r_{T_1}|_{W_1} \colon W_1 \to Y$. Set $\ell_{W_1}=\ell_Y\circ r_1$ and $f_{W_1}=f_Y\circ r_1$. Then $\ell_{W_1}$ is a
$\theta$-structure on $W_1$, with
$\theta\circ\ell_{W_1}=f_{W_1}$. The extended stable trivialization on $T_1$ restricts to a stable trivialization of $r_1^*\nu(Y)\oplus TW_1$. The map $r_1$ induces an isomorphism on $\pi_1$ and, being a retraction, is surjective on $\pi_2$. Thus $r_1$ is $2$-connected, and the inclusion $Y\hookrightarrow W_1$ is $1$-connected.

Next we proceed differently in the two cases. 
\begin{enumerate}
\item 
Let $n \ge 6$. Then we perform a second step.

\emph{Step 2:} We further need to kill $K_2 \coloneqq \ker \bigl((r_1)_* \colon \pi_2(W_1) \to \pi_2 (Y)\bigr)$. Since $r_1$ is a retraction, the long exact sequence of the pair $(W_1,Y)$ identifies $K_2 \cong \pi_2(W_1,Y).$ Since $(W_1,Y)$ is a finite CW pair and $r_1$ induces an isomorphism on $\pi_0$ and $\pi_1$, $\pi_2(W_1, Y)$ is finitely generated as a $\mathbb Z[\pi_1(Y)]$-module \cite{Wall1965Finiteness}. Since $2 < n/2$ we can represent each of those module generators $\beta$ by an embedding $\beta \colon S^2 \hookrightarrow \Int W_1$. Since $r_1 \circ \beta$ is contractible, there is a map $G_\beta \colon D^3 \to Y$ such that $G_\beta|_{S^2}=r_1\circ \beta$.

The stable trivialization of $r_1^*\nu(Y) \oplus TW_1$ induces a stable trivialization of $\beta^*r_1^*\nu(Y) \oplus \beta^* TW_1$. Since $r_1\circ\beta$ extends over $D^3$, the bundle $\beta^*r_1^*\nu(Y)$ is trivial. Thus $\beta^* TW_1 \cong \nu(S^2,W_1) \oplus TS^2$ is stably trivial.
Since $TS^2$ is stably trivial, it follows that the normal bundle of $\beta$ in $W_1$, $\nu(S^2,W_1)$, is stably trivial as well. Its rank is $n-2 \geq 4$, so it is trivial, and the compatible stable normal framing can be realized by an actual framing. Choose this framing and a tubular neighborhood $S^2\times D^{n-2}$ of $\beta$.

As in Step 1, we homotope $r_1$, fixing the central sphere and a neighborhood of $\partial W_1$, in the attaching neighborhood $S^2\times D^{n-2}$ to be constant in the normal-disk direction. We transport the stable trivialization through this homotopy. Extend $r_1$ over the $3$-handle $D^3\times D^{n-2}$ by $(z,v)\mapsto G_\beta(z)$. By the choice of framing, the stable trivialization extends over the trace. Composing the extended map with
$\ell_Y$ and $f_Y$ therefore gives a strict
$\theta$-structure and its stable tangent classifying map, respectively.

Hence we can kill $K_2$ in finitely many surgery steps. We end up with a $\theta$-cobordism relative to the boundary $T_2\colon W_1 \rightsquigarrow W_2$, such that $W_2$ is obtained from $W_1$ by 2-surgeries in the interior, and a retraction $r_2 \colon W_2 \to Y$ that induces isomorphisms on $\pi_1$ and $\pi_2$. Being a retraction, it is surjective on $\pi_3$. Thus $r_2$ is $3$-connected, and the inclusion $Y\hookrightarrow W_2$ is $2$-connected.

\vspace{1em}

Together $T_1$ and $T_2$ give a $\theta$-cobordism relative to the boundary $T \colon W \rightsquigarrow W_2$.

Now we take $W_2$ with the opposite direction, $W_2^{\op}$, and regard $W_2^{\op}$ as a cobordism from $\Sigma$ to $Y$. Since the inclusion $Y \hookrightarrow W_2^{\op}$ is 2-connected, and $\Sigma$ admits a metric of positive scalar curvature $g_\Sigma$, by Gajer's theorem on the trace of surgery \cite{Gajer1987Riemannian} (see also \cite[Remark 2.12 (iii)]{Frenck2026Doubling}), $W_2^{\op}$ admits a positive scalar curvature metric $g_{W_2^{\op}}$ which extends $g_\Sigma$ on $\Sigma$ and is of product type near both boundaries.

We let $C_2 \coloneqq W \cup_{\Sigma} W_2^{\op}$. Then the metric $g$ on $W$ and the metric $g_{W_2^{\op}}$ on $W_2^{\op}$ glue together smoothly to get a positive scalar curvature metric $g_{C_2}$ on $C_2$, such that $g_{C_2} = g$ on $W$ and $g_{C_2}$ is a product near the boundary $Y_+$ coming from $W_2^{\op}$.

Then by gluing together the trivial cobordism over $W$ and the $\theta$-cobordism relative to the boundary $T^{\op} \colon W^{\op} \rightsquigarrow W_2^{\op}$, we get a $\theta$-cobordism relative to the boundary from $W \cup_{\Sigma} W^{\op}$ to $C_2$.

By \cite[Proposition 2.13]{Frenck2026Doubling}, since $C_2$ admits a positive scalar curvature metric $g_{C_2}$, the cylinder $Y \times [0,1]$ admits a positive scalar curvature metric $\widetilde{g}$ such that $g_{C_2}$ and $\widetilde{g}$ agree in a neighborhood of both boundary components. This shows $\widetilde{g} = g$ near $Y \times \{0\}$ and $\widetilde{g}$ is a product near $Y \times \{1\}$.

\vspace{1em}

\item Let $n \ge 5$. Then no second step as above is needed. Instead, since $\Sigma$ is minimal, by \cite[Proposition 3.8]{BarHanke2023BoundaryConditions}, we can take a new PSC metric $g'$ on $W$, agreeing with $g$ outside a collar of $\Sigma$, which makes $\Sigma$ strictly mean convex as a boundary of $W$.
In particular, $g'=g$ near $Y\times\{0\}$.

Regard $W_1^{\op}$ as a cobordism from $\Sigma$ to $Y$.
Since the inclusion $Y\hookrightarrow W_1^{\op}$ is
$1$-connected and $n\geq5$, handle cancellation gives
a handle decomposition of $W_1^{\op}$ relative to $\Sigma$ with all handles of index at most $n-2$ by a result of Wall \cite[Theorem 3]{Wall1971Geometrical} (see also \cite[Theorem 2.16]{Frenck2026Doubling}).
Thus all these handles have codimension at least $2$.

Set $C_1 \coloneqq W\cup_\Sigma W_1^{\op}.$ Then by \cite[Corollary 2.15]{Frenck2026Doubling} and its proof, since $g'$ is a metric of positive
scalar curvature on $W$ with strictly mean convex boundary $\Sigma$, we obtain a positive scalar curvature metric $g_{C_1}$ on $C_1$ which agrees with $g'$ near the incoming boundary $Y_-$ from $W$ and makes the outgoing boundary $Y_+$ from $W_1^\op$ strictly mean convex.

As in the first case, gluing the trivial cobordism over $W$ to the $\theta$-cobordism $T_1^{\op}\colon W^{\op}\rightsquigarrow W_1^{\op}$ along the common vertical face $\Sigma\times[0,1]$ gives a $\theta$-cobordism relative to the boundary from $W\cup_\Sigma W^{\op}$ to $C_1$.
By \cite[Proposition 2.13]{Frenck2026Doubling}, there is therefore a positive scalar curvature metric $\widetilde{g}'$ on $Y\times[0,1]$ agreeing with $g_{C_1}$ near both boundary components. Consequently, $\widetilde{g}'=g_{C_1} = g' = g$ near $Y\times\{0\}$, and $Y\times\{1\}$ is strictly mean convex.

Finally, using \cite[Theorem 3.7]{BarHanke2023BoundaryConditions}, we can further deform $\widetilde{g}'$ in an arbitrarily small collar of $Y\times\{1\}$, preserving positive scalar curvature, so that $Y\times\{1\}$ becomes totally geodesic, and hence minimal. The resulting metric $\widetilde{g}$ still agrees with $g$ near $Y\times\{0\}$, as desired.
\end{enumerate}
\end{proof}

Now we are ready to give the proofs of Theorems \ref{thm: UPSC dim 6, 7} and \ref{thm: PSC dim 5, 6, 7}.
\begin{proof}[Proof of Theorem \ref{thm: UPSC dim 6, 7}]
Let $n=6,7$. Suppose that $\Int X$ admits a complete metric of uniformly positive scalar curvature. Then by Theorem \ref{thm: Wang--Zhu} and taking the connected components containing a fixed basepoint, $\Int X$ admits an exhaustion by compact connected smooth codimension-zero submanifolds $K_1 \subset K_2 \subset \dots \subset \Int X$ and a complete metric $g$ of uniformly positive scalar curvature such that near each hypersurface $\Sigma_i \coloneqq  \partial K_i$, the metric is of product form $g = dt^2 + g_{\Sigma_i}$. Set $g_i=g|_{K_i}$.

Break $\partial X$ into its connected components as $\partial X=\sqcup_{\alpha=1}^m Y_\alpha$, where each $Y_\alpha$ is connected. For each $Y_\alpha$, fix a collar and a parametrization $C_\alpha\cong  Y_\alpha \times [0,\infty)$. Let
$X_0=\Int X\setminus \sqcup_{\alpha=1}^m \bigl(Y_\alpha\times(0,\infty)\bigr)$ be the remaining compact core.
Choose $i$ sufficiently large that $X_0\subset\Int K_i$.

For each $\alpha$, set
\[
W_\alpha=K_i\cap C_\alpha,
\qquad
\Sigma_\alpha=\partial K_i\cap C_\alpha.
\]
Then $W_\alpha$ is connected, since $K_i$ and $Y_\alpha$ are connected and $X_0\subset\Int K_i$, and
\[
\partial W_\alpha
=(Y_\alpha\times\{0\})\sqcup\Sigma_\alpha.
\]
The restriction of $g_i$ to $W_\alpha$ has positive scalar curvature and is a product near $\Sigma_\alpha$.
By Lemma \ref{lem: psc product} \eqref{item:n>=6}, there is a positive scalar curvature metric $\widehat g_\alpha$ on $Y_\alpha\times[0,1]$ agreeing with $g_i$ near $Y_\alpha\times\{0\}$ and equal to a product near $Y_\alpha\times\{1\}$.

Glue these cylinders to $(X_0,g_i|_{X_0})$ along $Y_\alpha\times\{0\}$. The metrics agree near the gluing hypersurfaces. The resulting manifold is diffeomorphic to $X$, with a smooth metric of positive scalar curvature which is a product near the boundary.
\end{proof}

\begin{proof}[Proof of Theorem \ref{thm: PSC dim 5, 6, 7}]
Let $n = 5,6,7$. The proof proceeds the exact same way as the proof of Theorem \ref{thm: UPSC dim 6, 7}, with Proposition \ref{prop: PSC deformation} replacing Theorem \ref{thm: Wang--Zhu} and Lemma \ref{lem: psc product} \eqref{item:n>=5} replacing Lemma \ref{lem: psc product} \eqref{item:n>=6}.
\end{proof}

\subsection{Proof of Theorem \ref{thm: dim 5 homeo}}
\label{subsec: dim 5 homeo}

We first record the following version of the surgery theorem that will be needed:
\begin{proposition}\label{prop: surgery}
Let $n \ge 3$. Let $M^n$ be a closed smooth manifold admitting a positive scalar curvature metric. Suppose $Q\subset M$ is a compact codimension-zero submanifold admitting a handle decomposition with all handles of index at most $n-3$. Then $M\setminus\operatorname{Int}Q$ admits a positive scalar curvature metric that is a product near its boundary.    
\end{proposition}
\begin{proof}
Equip the $0$-handles of $Q$ with standard positive scalar curvature metrics that are products near their boundaries. Since all remaining handles have index at most $n-3$, Gajer's theorem on the trace of surgery \cite{Gajer1987Riemannian} extends these metrics to a positive scalar curvature metric on $Q$ that is a product near its boundary.

Take $$W=\bigl(M\times[0,1]\bigr)\setminus
\operatorname{Int}\bigl(Q\times[1/2,1]\bigr),$$ 
rounding corners compatibly with the boundary collars. Then $W$ gives a relative cobordism from $M\sqcup Q^{\op}$ to $M\setminus\Int Q$. Relative to the incoming boundary $M\sqcup Q^{\op}$,
the cobordism has a $(k+1)$-handle for each
$k$-handle of $Q$.

We can view $W$ as the trace of interior surgeries from $M\sqcup Q^{\op}$ to $M\setminus\Int Q$, all of codimension at least $3$. Also note they are in the interior and away from the boundary. Then we can apply the surgery theorem of Gromov--Lawson \cite{GromovLawson1980Classification} and get a metric of positive scalar curvature on $M\setminus\Int Q$. The metric near the boundary is unchanged and is thus still a product near the boundary. 
\end{proof}

We next get some topological control on $\partial X$. 
\begin{lemma}\label{lem: boundary}
Let $(X, \partial X)$ be a compact connected oriented smooth $5$-manifold with non-empty boundary $\partial X$. Assume each component of $\partial X$ is simply connected. Suppose $\Int X$ admits a complete metric of uniformly positive scalar curvature. Let $\Sigma$ be a component of $\partial X$. Then
$$\Sigma \cong_{\mathrm{homeo}} \begin{cases}
\#_k(S^2\times S^2)  & \text{for some $k \ge 0$ if $\Sigma$ is spin,} \\
\#_{p}\CP^2\#_{q}\overline{\CP}^2 & \text{for some $p, q \ge 0$, $p+q >0$ if $\Sigma$ is nonspin.} \\
\end{cases}
$$
When $k = 0$, $\#_0(S^2\times S^2)$ denotes $S^4$.
\end{lemma}
\begin{proof}
Suppose that $\Int X$ admits a complete metric of uniformly positive scalar curvature. Let $\Sigma$ be a component of $\partial X$. Then by Theorem \ref{thm: Wang--Zhu}, we can get a separating hypersurface with a metric of positive scalar curvature in a cylinder diffeomorphic to $\Sigma \times [0,1]$. By \cite[Proposition 31]{KumarSen2025ExoticStructures}, there exists at least one smooth structure on $\Sigma$ that admits a positive scalar curvature metric.

If $\Sigma$ is nonspin, then $\Sigma$ has an odd intersection form $Q_\Sigma$. Let $p = b_2^+(\Sigma)$ and $q = b_2^-(\Sigma)$. If $Q_\Sigma$ is definite (i.e. one of $p,q$ is $0$), then since $\Sigma$ is smooth, Donaldson's diagonalization theorem \cite{Donaldson1983GaugeTheory} implies that $Q_\Sigma = p \langle 1 \rangle \oplus q \langle -1 \rangle$. If $Q_\Sigma$ is indefinite (i.e., $p,q >0$), the algebraic classification of indefinite odd unimodular forms also gives $Q_\Sigma = p \langle 1 \rangle \oplus q \langle -1 \rangle$. Consequently, by Freedman's theorem \cite{Freedman1982Topology}, $\Sigma$ is homeomorphic to $\#_{p}\CP^2\#_{q}\overline{\CP}^2$. 

If $\Sigma$ is spin, then by \cite[Observation 30]{KumarSen2025ExoticStructures}, $\hat{A}(\Sigma) = 0$. In dimension 4 $\hat{A}(\Sigma) = - \sigma(\Sigma)/8$. Thus $\Sigma$ is a closed simply connected spin 4-manifold with zero signature, and by Freedman's theorem \cite{Freedman1982Topology}, $\Sigma$ is homeomorphic to $\#_{k} (S^2 \times S^2)$ for some $k \ge 0$. 
\end{proof}

We construct a cap in the totally nonspin case that will be used later in the proof.
\begin{lemma}\label{lem: cap}
Let $\Sigma =\sqcup_{j=1}^{m}\Sigma_j$ be a closed oriented $4$-manifold, where
$$\Sigma_j\cong_{\mathrm{diffeo}} \#_{p_j}{\CP}^{2} \#_{q_j}\overline{\CP}^{2},
$$
and $p_j + q_j >0$.
Suppose $\sigma(\Sigma)=0$. Then $\Sigma$ bounds a compact, connected, simply connected, oriented $5$-manifold $Q$ admitting a positive scalar curvature metric that is a product near the boundary.
\end{lemma}
\begin{proof}
Since $0 = \sigma(\Sigma)=\sum_j p_j - \sum_j q_j$, let $P\coloneqq  \sum_j p_j$. Take $P$ copies of the cylinder ${\CP}^{2}\times[0,1]$, equipped with the product of the Fubini--Study metric and $dt^2$. The boundary of each cylinder gives one copy of $\CP^2$ and one copy of $\overline{\CP}^2$. 

Perform a boundary connected sum of these cylinders by attaching $5$-dimensional $1$-handles along boundary $4$-balls of these cylinders to assemble the boundary into the prescribed components $\Sigma_j$. These handle attachments realize a sequence of codimension-$4$ surgeries on the boundary, so we can extend the positive scalar curvature metric across the handles while retaining a product structure near the new boundary by Gajer's theorem \cite{Gajer1987Riemannian}.
The resulting $5$-manifold $Q_0$ obtained by attaching the 1-handles to the cylinders thus has $\Sigma$ as its boundary, and admits a metric of positive scalar curvature that is a product near the boundary. 

We perform interior connected sums on $Q_0$ to get a connected manifold, and then kill the fundamental group by interior 1-surgeries on embedded circles representing the generators. This results in a connected, simply connected 5-manifold $Q$ with the same boundary $\Sigma$. $Q$ is obtained from $Q_0$ by $0$- and $1$-surgeries, so the surgery theorem of Gromov--Lawson \cite{GromovLawson1980Classification} implies that $Q$ still admits a metric of positive scalar curvature that is product near the boundary.
\end{proof}

We also record an elementary construction for later use. 
\begin{lemma}\label{lem: star}
Let $W$ be a compact connected manifold with $m \ge 1$ boundary components. Write
$$\partial W=\Sigma_1\sqcup \Sigma_2\sqcup\cdots\sqcup \Sigma_m,
$$
and distinguish $\Sigma_1$ as the root boundary component. Let $S$ denote the star graph on $m$ vertices $\{v_1, \dots, v_m\}$ with center at $v_1$. Then there exists a continuous map $f:W\to S$ such that some collar neighborhood of $\Sigma_j$ is mapped constantly to $v_j$ for each $j$.   
\end{lemma}
\begin{proof}
Choose pairwise disjoint collar neighborhoods $N_j \cong \Sigma_j\times[0,2]$ for each boundary component $\Sigma_j$.
For each $j\ge 2$, let $e_j\subset S$ denote the edge joining $v_j$ to the central vertex $v_1$, and choose a parametrization $\gamma_j:[0,1]\to e_j$
with $\gamma_j(0)=v_j$, $\gamma_j(1)=v_1$.

For each $j\ge2$, define $f$ on $\Sigma_j\times[0,2]$ by
$$f(x,t)=
\begin{cases}
v_j, & 0\le t\le 1 \\
\gamma_j(t-1), & 1\le t\le2.
\end{cases}
$$
Define $f\equiv v_1$ on the complement of these $m-1$ collars. $f$ is a continuous map that has the required collar restrictions.
    
\end{proof}

We are now ready to prove Theorem \ref{thm: dim 5 homeo}.
\begin{proof}[Proof of Theorem \ref{thm: dim 5 homeo}]
Take $X$ as in the statement. Then by Theorem \ref{thm: PSC dim 5, 6, 7},  $X$ admits a metric $g$ of positive scalar curvature with minimal boundary. 
Let $Y\coloneqq  \partial X$, and let $m$ be the number of components of $Y$. Let the components of $Y$ be $Y_1, \dots, Y_m$. By Lemma \ref{lem: boundary}, each $Y_j$ is homeomorphic to either $\#_{k_j} (S^2 \times S^2)$ or $\#_{p_j}\CP^2\#_{q_j}\overline{\CP}^2$.
Let $Y_j'$ be the standard $\#_{k_j} (S^2 \times S^2)$ or $\#_{p_j}\CP^2\#_{q_j}\overline{\CP}^2$ homeomorphic to $Y_j$.

Since $Y_j$ and $Y_j'$ are simply connected and homeomorphic, there exists a smooth $h$-cobordism $C_j\colon Y_j\rightsquigarrow Y_j'$ \cite[Theorem 2]{Wall1964SimplyConnected}. Then the topological $h$-cobordism theorem proven by Freedman \cite[Theorem 1.3]{Freedman1982Topology} implies $C_j$ is homeomorphic to a product $Y_j \times [0,1]$. 

Let $C\coloneqq  \sqcup_j C_j$. 
Take $X'\coloneqq  X \cup_{Y} C$. Then $X'$ is homeomorphic to $X$. Our goal is to show $X'$ admits a metric of positive scalar curvature which is a product near the boundary. To do this, we will cap it off by suitable caps to get a closed 5-manifold $N$, and show it has a metric of positive scalar curvature using Theorem \ref{thm: psc bordism}. The caps will be nice enough so that the surgery theorem allows us to remove the caps and get the desired metric on $X'$.

If $Y_j$ is spin, then $Y_j' = \#_{k_j} (S^2 \times S^2)$, and we can take $(Y_j')^{\op}$ as the boundary of a boundary connected sum $Q_j\coloneqq  \natural_{k_j}(D^3\times S^2)$. Here $\natural_0(D^3\times S^2)=D^5$. $Q_j$ is spin since both $D^3$ and $S^2$ are, and boundary connected sum preserves spin structures. In its handle decomposition, $Q_j$ consists of a single $0$-handle together with $k_j$ handles of index $2$. 
Equip the $0$-handle with the standard positive scalar curvature metric that is a product near the boundary. Gajer's theorem \cite{Gajer1987Riemannian} then extends this metric to a positive scalar curvature metric on $Q_j$ that is a product near its boundary.

If $Y_j$ is nonspin, then $Y_j' = \#_{p_j}\CP^2\#_{q_j}\overline{\CP}^2$. 
Its signature is $\sigma(Y_j') = p_j - q_j$. Thus $Y_j'$ does not bound unless $p_j = q_j$. However, since $Y$ bounds $X$, $0 = \sigma (Y) = \sum \sigma(Y_j) = \sum \sigma(Y_j')$. Since the signature of $\#_{k} (S^2 \times S^2)$ is zero, the sum of signatures of those nonspin $Y_j'$ is zero. 

Take the double $DX = X \cup_Y X^{\op}$. Since $(X,g)$ has positive scalar curvature with minimal boundary, $DX$ admits a metric of positive scalar curvature \cite[Theorem 5.7]{GromovLawson1980Spin}. 

We now proceed depending on whether $X$ is spin or totally nonspin.

\emph{Case 1:} Assume $X$ is spin. Then $\partial X$ and $DX$ are spin as well. Let $U_j \coloneqq  C_j \cup Q_j$. The double $DU_j$ is also spin. 
We consider the closed manifold
$$N\coloneqq  X \cup_{Y_j} U_j =X' \cup_{Y_j'} Q_j$$
where the union is taken over all $j$'s. $N$ is a spin manifold by our construction. See Figure \ref{fig: spin} for a visualization. 

\begin{figure}[ht]
    \centering
    \includegraphics[width= 0.9\linewidth]{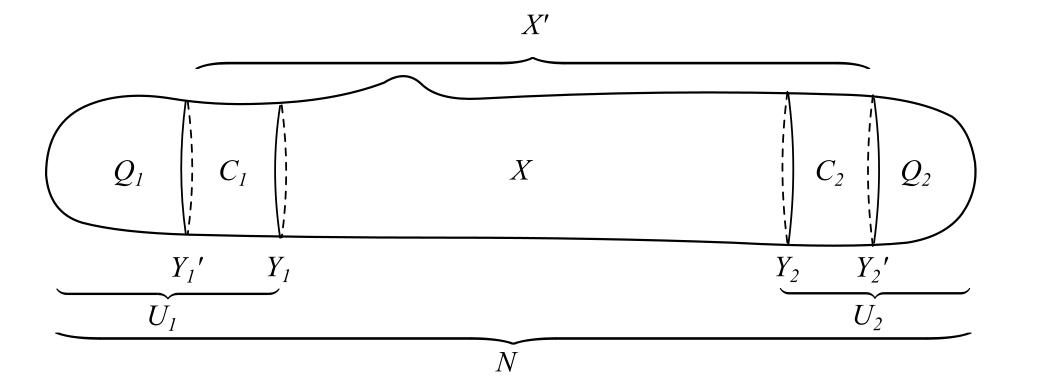}
    \caption{The manifold $N$ in the spin case.}
    \label{fig: spin}
\end{figure}

Since each $Y_j'$ and each $Q_j$ is simply connected, by van Kampen's theorem, $\pi_1(N) = \pi_1 (X) =: \pi$.

Take a classifying map $u_0 \colon N \to B\pi$. Since each $C_j \cup_{Y_j'} Q_j$ is simply connected, $u_0 |_{C_j \cup_{Y_j'} Q_j}$ is null-homotopic. By the homotopy extension property, we can fix a point $b \in B\pi$ and homotope $u_0$ to another classifying map $u \colon N \to B\pi$ such that for each $j$, $u|_{C_j \cup_{Y_j'} Q_j}$ is the constant map to $b$. We denote by $b$ the constant map to $b$.
Then we can extend $u|_X$ to a map $\hat{u} \colon DX \to B \pi$ by setting $\hat{u} = u$ on $X$ and extending by the constant map $\hat{u} \equiv b$ on $X^{\op}$.

Then we get the following bordism equivalence in $\Omega_5^{Spin}(B\pi)$:
$$[N, u] + [N^{\op}, b] = [DX, \hat{u}] + \sum_{j=1}^m [DU_j, b].$$

Since $\Omega_5^{Spin}(*) = 0$ \cite{Milnor1963Spin}, we have $[N^{\op}, b] = 0 = [DU_j, b]$, so 
$$[N, u] = [DX, \hat{u}] \in \Omega_5^{Spin}(B\pi).$$
Since $DX$ admits a metric of positive scalar curvature, Theorem \ref{thm: psc bordism} (1) implies that $N$ admits a metric of positive scalar curvature. 

Since each $Q_j$ consists of only $0$ and $2$-handles, by Proposition \ref{prop: surgery}, $X' = N \setminus \sqcup_j \Int Q_j$ admits a metric of positive scalar curvature that is a product near the boundary as desired.
\vspace{1em}

\emph{Case 2:} Assume $X$ is totally nonspin. Let $I$ denote the set of indices $j$ where $Y_j'$'s are spin, and let $J$ denote the set of indices $j$ where $Y_j'$'s are nonspin.

For each $j \in I$, we still cap off $Y_j'$ by $Q_j$. Define $U_j \coloneqq  C_j \cup Q_j$.

For the nonspin components, $\sqcup_{j \in J} Y_j'$ has signature zero, so $\sqcup_{j \in J} (Y_j')^{\op}$ bounds a compact 5-manifold $\hat{Q}$ as constructed in Lemma \ref{lem: cap}. If $J = \varnothing$ then we just take $\hat{Q} =\varnothing$. When $J\neq\varnothing$, $\hat{Q}$ is connected, simply connected, and admits a metric of positive scalar curvature that is product near the boundary. Define $\hat U\coloneqq  \cup_{j \in J} C_j \cup_{Y_j'} \hat{Q}$.

We consider the closed manifold
$$N\coloneqq  X \cup_{Y_j, j \in I} U_j \cup_{Y_j, j \in J} \hat{U} = X' \cup_{Y_j', j \in I} Q_j \cup_{Y_j', j \in J} \hat{Q}.$$
See Figure \ref{fig: totally nonspin} for a visualization. Because $X$ is totally nonspin, $N$ is also totally nonspin by our construction.

\begin{figure}[ht]
    \centering
    \includegraphics[width= 0.9\linewidth]{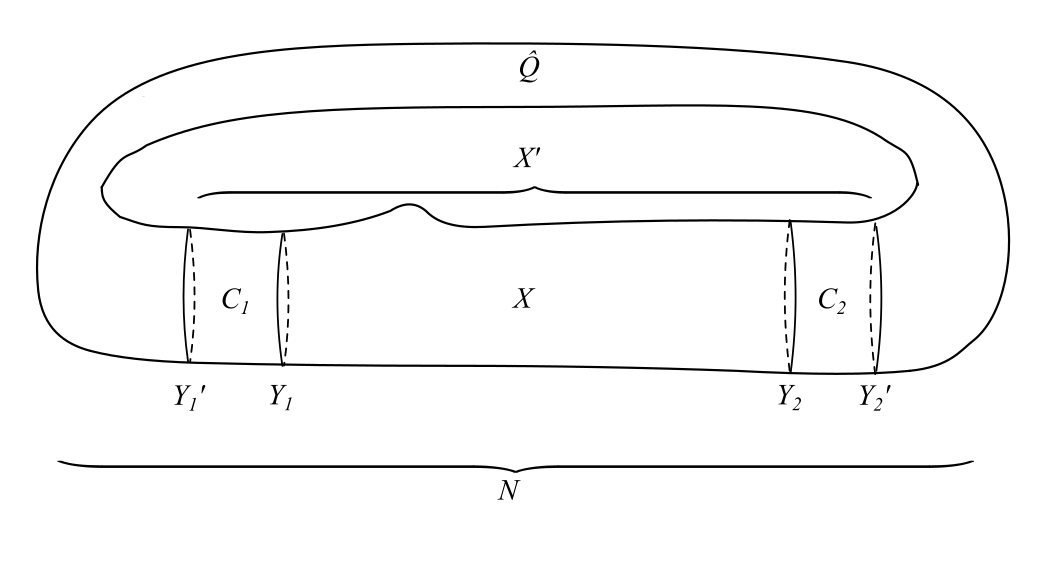}
    \caption{The manifold $N$ in the totally nonspin case. In the case drawn, $|I| = 0$ and $|J|= 2$.}
    \label{fig: totally nonspin}
\end{figure}

Denote $\pi\coloneqq  \pi_1 (X) = \pi_1(X')$.
By van Kampen's theorem, since each $Q_j$ and each $Y_j$ for $j \in I$ are connected and simply connected, attaching them to $X'$ does not change the fundamental group. 

We first assume $J \neq \varnothing$. $\hat{Q}$ is connected and simply connected, but it is attached to $X'$ over $|J|$ simply connected components, so by van Kampen's theorem for fundamental groupoids,
$$\pi_1(N) = \pi_1(X') * F_{|J|-1} = \pi * F_{|J|-1} \coloneqq \Gamma,$$
where $F_{|J|-1}$ is the free group on $|J|-1$ elements.

Notice that 
$$B\Gamma\simeq B\pi\vee\bigvee^{|J|-1}S^1.$$

Let $b$ denote the wedge point where all the wedge sums take place. Let $Bi \colon B\pi \to B\Gamma$ be the inclusion map, and $Br: B\Gamma \to B\pi$ be the map where all $S^1$'s are collapsed to the wedge point.  $Br \circ Bi =\id$. For simplicity of notation, we identify $\Omega_5^{SO}(B\pi)$
with its image in $\Omega_5^{SO}(B\Gamma)$ under $(Bi)_*$. Similarly we identify $\Omega_5^{SO}(*)$ with its image in $\Omega_5^{SO}(B\pi)$ and $\Omega_5^{SO}(B\Gamma)$ by mapping $*$ to $b$.

We can compute 
\begin{align*}
\Omega_5^{SO}(B\Gamma) & \cong\Omega^{SO}_5 (*) \oplus \widetilde \Omega_5^{SO}(B\Gamma) & \text{ by \eqref{eq: reduced bordism}}\\
& \cong \Omega^{SO}_5 (*) \oplus \widetilde \Omega_5^{SO}(B\pi) \oplus (\bigoplus^{|J|-1} \widetilde \Omega_5^{SO}(S^1))  & \text{ by \eqref{eq: additivity}}\\
& \cong  \Omega_5^{SO}(B\pi) \oplus (\bigoplus^{|J|-1} \widetilde \Omega_5^{SO}(S^1))  & \text{ by \eqref{eq: reduced bordism}.}
\end{align*}
Recall from \eqref{eq: bordism S^1} that $\widetilde\Omega_5^{SO}(S^1) \cong \Z$ and $[\CP^2 \times S^1, \operatorname{pr}_{S^1}]$ generates the group $\widetilde\Omega^{SO}_5(S^1)$. 

$(Br)_* \colon \Omega_5^{SO}(B\Gamma) \to \Omega_5^{SO}(B\pi)$ is projection onto the first summand of the above decomposition. Then the kernel of $(Br)_*$ is generated by
\begin{align}\label{eq: P_i}
P_i=
[\CP^2\times S^1,\,
\ell_i\circ\operatorname{pr}_{S^1}],    
\end{align}
where \(\ell_i:S^1\to B\Gamma\) represents the \(i\)-th free generator. The standard product metric of $\CP^2\times S^1$ has positive scalar curvature. 

Take a classifying map $u_0\colon X \to B\pi$. Similar to the spin case, since the boundary components are simply connected, we can take $u_0$ to equal the constant $b$ on the boundary collars. Extend it constantly to get a map 
$u: N \to B\pi$. 

On the other hand, we can also build a classifying map $c\colon N \to B\Gamma$ from $u_0$. To get $c$ from $u_0$, we still extend by constant $b$ to all the $Q_j$'s for $j \in I$. To extend to $\hat Q$, fix one $j \in J$ as the center of the star, by Lemma \ref{lem: star}, we can construct a continuous map from $\hat U$ to the star graph on $|J|$ vertices. Gluing all the vertices together yields a continuous map from $\hat{U}$ to $\bigvee_{|J|-1}S^1$. Gluing this map with the map $u_0: X \to B\pi$ yields the map $c: N \to B\Gamma$. Under the identification $\pi_1(N)\cong\pi*F_{|J|-1}$, the induced map $c_*$ sends the subgroup $\pi_1(X)\cong\pi$ isomorphically onto the $\pi$-factor and the additional free generators to the corresponding circle generators in $B\Gamma$, so $c$ is a classifying map. By our construction, $[N, u] = (Br)_* [N, c] \in \Omega_5^{SO}(B\pi)$. Our goal is to show $[N,c] \in \Omega_5^{SO,+}(B\Gamma)$.

We also extend $u_0$ to a map $\hat{u} \colon DX \to B \pi$ by letting $\hat{u} = u$ on $X$ and extend by the constant map $\hat{u} \equiv b$ on $X^{\op}$.

We get the bordism equivalence in $\Omega_5^{SO}(B\pi)$:
$$[N, u] + [N^{\op}, b] = [DX, \hat{u}] + \sum_{j \in I} [DU_j, b] + [D\hat{U}, b].$$

Now since $D\hat{U}$ bounds $\hat{U} \times [0,1]$ after rounding corners and the constant map $b$ extends over $\hat{U} \times [0,1]$, we have $[D\hat{U}, b] = 0$. Similarly $[DU_j, b] = 0$.
Thus
$$[N, u] - [DX, \hat{u}] = [N, b].$$

Since $b$ is the constant map, $[N, b] \in \Omega_5^{SO}(*)$. As shown by \cite[Theorem 3]{GromovLawson1980Classification}, every class in $\Omega_5^{SO}(*)$ has a simply connected representative carrying a metric of positive scalar curvature. Let $(M, b)$ be such a PSC representative in the class of $[N, b]$. 

Thus
$$[N, u] = [DX, \hat{u}] + [M, b] \in \Omega_5^{SO}(B\pi).$$

To show $[N,c] \in \Omega_5^{SO}(B\Gamma)$ has a PSC representative, we consider the class $[N,c] - [N, u] \in \Omega_5^{SO}(B\Gamma)$. Since  $[N, u] = (Br)_* [N, c] \in \Omega_5^{SO}(B\pi)$, we find $[N,c] - [N, u] \in \ker((Br)_* \colon \Omega_5^{SO}(B\Gamma) \to \Omega_5^{SO}(B\pi))$.

Since $\ker(Br)_*$ is generated by the $P_i$'s as defined in \eqref{eq: P_i}, we can find $a_i \in \Z$ such that 
$$[N,c] - [N, u] = \sum_{i=1}^{|J|-1}a_i P_i \in \Omega_5^{SO}(B\Gamma).$$

Therefore we have in $\Omega_5^{SO}(B\Gamma)$ that
$$[N,c] = [N, u] + \sum_{i=1}^{|J|-1}a_i P_i = [DX, \hat u] + [M,b] + \sum_{i=1}^{|J|-1}a_i P_i.$$

By our choice, each representative on the right-hand side has a metric of positive scalar curvature, so the class $[N,c]$ has a PSC representative. Since $N$ is totally nonspin, by Theorem \ref{thm: psc bordism} (2), $N$ admits a metric of positive scalar curvature.

The above argument works when $J \neq \varnothing$. If $J=\varnothing$, then we simply have $\pi_1(N) = \pi_1(X') = \pi$. The same argument as above applies with $\hat Q$ omitted, $\Gamma=\pi$, and $c=u$, and we get the bordism equivalence $[N,c] = [DX, \hat u] + [M,b]$ in $\Omega_5^{SO}(B\Gamma)$, which again shows $N$ admits a metric of positive scalar curvature by Theorem \ref{thm: psc bordism} (2).

Lastly, we show that $X'$ admits the desired metric of positive scalar curvature that is a product near the boundary. Let $Q\coloneqq  \cup_{j \in I}  Q_j \cup \hat{Q}$. Then $N = X' \cup Q$. So far, we have shown that $N$ has a metric of  positive scalar curvature, and $Q$ admits a metric of positive scalar curvature that is a product near the boundary.

There is an oriented bordism from $N\sqcup Q^{\op}$ to $X'$ relative to their identified boundary. Explicitly, take
\[
W= N \times [0,1]\setminus \operatorname{Int}\bigl(Q\times[1/2,1]\bigr)
\]
rounding corners compatibly with the boundary collars.
Its incoming boundary is $(N\times\{0\})\sqcup(Q^{\op}\times\{\frac12\})$, its outgoing boundary is $X'\times\{1\}$, and its vertical
boundary is $\partial X'\times[\frac12,1]$.

The map $u\circ\operatorname{pr}_N\colon W\to B\pi$ restricts to $u$ on $N$, to the constant map $b$ on $Q^{\op}$ and the vertical boundary, and to $u|_{X'}$ on $X'$. By construction, $u|_{X'}$ classifies the universal cover of $X'$. Thus $W$ is a bordism over $B\pi$ relative to the boundary.

Since $X'$ is totally nonspin and has dimension $5$, by the version of the surgery theorem for manifolds with boundary in \cite[Lemma 2.6]{Frenck2023Nonspin} (cf. \cite[Theorem 1.5]{EbertFrenck2021Surgery} and \cite[Remark 2.12(i)]{Frenck2026Doubling}), the manifold $X'$ admits a metric of positive scalar curvature which restricts to the same metric near the boundary as $N\sqcup Q^{\op}$. Thus $X'$ admits a metric of positive scalar curvature that is a product near the boundary as desired.

\end{proof}

\bibliographystyle{alpha}
\bibliography{main}
\end{document}

%% file: conditions.tex
\begin{tikzpicture}[
  node distance=2.2cm and 2.8cm,
  every node/.style={align=center}
]
\node (prod) {
  (1) PSC on $X^n$,\\
  product near $\partial X$
};

\node (min) [right=of prod] {
  (2) PSC on $X^n$,\\
  $\partial X$ minimal
};

\node (upsc) [below=of prod] {
  (3) complete UPSC\\
  on $\operatorname{Int}X$
};

\node (psc) at (min |- upsc) {
  (4) complete PSC\\
  on $\operatorname{Int}X$
};

\draw[implication] (prod) -- (min);
\draw[implication, bend left=12] (prod) to (upsc);
\draw[implication, dashed, bend left=12]
  (upsc) to node[left,font=\footnotesize] 
  {$n=3$: true (Gromov, Wang) \\
  $n=5$: not in general, \\
  some cases upto homeo. \\
  (Theorem \ref{thm: dim 5 homeo}) \\
  $n= 6, 7$: true (Theorem \ref{thm: UPSC dim 6, 7})
  } (prod);
\draw[implication] (upsc) -- (psc);
\draw[implication, dashed]
  (psc) to node[font=\footnotesize] {$n=3$: true (Schoen--Yau) \\
  $n=5,6,7$: true (Theorem \ref{thm: PSC dim 5, 6, 7})} (min);
\end{tikzpicture}